\documentclass[12pt]{amsart}

\usepackage{latexsym}
\usepackage{amsmath}
\usepackage{amssymb}
\usepackage{amscd}
\usepackage{multicol}
\usepackage[cmtip,matrix,arrow]{xy}
\usepackage{amsmath,amssymb}
   \usepackage{epsfig}
   
\numberwithin{equation}{section}

\newtheorem{theo}{\bf Theorem}[subsection]
\newtheorem{theorem}{\bf Theorem}[section]
\newtheorem{corollary}[theorem]{\bf Corollary}
\newtheorem{proposition}[theorem]{\bf Proposition}
\newtheorem{propo}[theo]{\bf Proposition}
\newtheorem{lemma}[theorem]{\bf Lemma}
\newtheorem{lem}[theo]{\bf Lemma}
\newtheorem{definition}[theo]{\bf Definition}
\newtheorem{definition-theorem}[theorem]{\bf Theorem-Definition}
\theoremstyle{remark}
\newtheorem{rem}[theorem]{Remark}

\def\ep{\epsilon}
\def\eg{\mathfrak{e}}
\def\fg{\mathfrak{f}}
\def\O{\mathbb{O}}
\def\H{\mathbb{H}}

\def\bR{\mathbb{R}}
\def\bC{\mathbb{C}}
\def\bZ{\mathbb{Z}}
\def\bQ{\mathbb{Q}}

\def\t{\mathfrak{t}}
\def\g{\mathfrak{g}}
\def\ga{\gamma}
 \def\k{\mathfrak{k}}
\def\quott({/\! /}

\def\g{\frak{g}}
\def\t{\frak{t}}

\def\a{\frak{a}}
\def\d{\frak{d}}

\def\k{\frak{k}}
\def\s{\frak{s}}
\def\m{\frak{m}}
\def\n{\frak{n}}
\def\h{\frak{h}}

\def\L{{\mathcal L}}

\def\S{{\mathcal S}}
\def\E{{\mathcal E}}

\def\tb{\tilde{b}}
\title[Topology of the  octonionic flag manifold]
{Topology of the  octonionic flag manifold}

\author[A.-L.~Mare ]{}
\author[M.~Willems]{}
\date{}

\begin{document}

\vspace{-0.5cm}

\maketitle

\vspace{-0.5cm}

\begin{center}

\begin{tabular}{ll }
{\sc Augustin-Liviu Mare} & {\sc Matthieu Willems } \\
Department of Mathematics and Statistics $ \ \ $ & D\'epartement d'informatique    \\
University of
Regina, College West 307.14 &  Universit\'e du Qu\'ebec \`a Montr\'eal \\
Regina, Saskatchewan, S4S 0A2 Canada & Montr\'eal, Qu\'ebec, H3C 3P8 Canada\\
{\tt mareal@math.uregina.ca} & {\tt  matthieu.willems@polytechnique.org}
 \end{tabular}

\end{center}

\vspace{0.3cm}

\begin{abstract}
The octonionic flag manifold ${\rm Fl}(\O)$ 
is the space of all pairs in   $\O {\rm P}^2\times \O {\rm P}^2$
(where $\O {\rm P}^2$ denotes the octonionic projective plane)
which satisfy a certain ``incidence" relation. 
It comes equipped with the projections 
$\pi_1,\pi_2 : {\rm Fl}(\O)\to \O {\rm P}^2$, which are $\O {\rm P}^1$-bundles,
as well as with an action of the group ${\rm Spin}(8)$.
The first two results of this paper give Borel type descriptions of  the usual, respectively
${\rm Spin}(8)$-equivariant cohomology of ${\rm Fl}(\O)$ in terms of $\pi_1$
and $\pi_2$ (actually the Euler classes of the tangent spaces to the fibers
of $\pi_1$, respectively $\pi_2$,
which are rank 8 vector bundles on ${\rm Fl}(\O)$).
We then obtain a Goresky-Kottwitz-MacPherson type description of 
the ring $H^*_{{\rm Spin}(8)}({\rm Fl}(\O))$. Finally, we consider the ${\rm Spin}(8)$-equivariant
$K$-theory ring of ${\rm Fl}(\O)$ and obtain  a Goresky-Kottwitz-MacPherson type description of this ring.

\vspace{0.2cm}

\end{abstract}

\vspace{-0.5cm}

\tableofcontents

\section{Introduction}\label{secfirst}
Let $\O$  denote the (normed, unital, non-commutative, and
non-associative) algebra of octonions and let $\O {\rm P}^2$ be the 
octonionic projective plane (see for instance \cite{Ba}, \cite{Fr}, and \cite{Mu}).
This space is an important example in  incidence geometry.
It turns out that there exists a natural identification
between the space of lines in $\O {\rm P}^2$ and $\O {\rm P}^2$ itself.
The octonionic flag manifold  ${\rm Fl}(\O)$ is the space of all pairs
$(p,\ell) \in \O {\rm P}^2\times \O {\rm P}^2$, where $p$ is a point and
$\ell$ a line, such that $p$ and $\ell$ are incident (see Definition \ref{defin} below).    
Both $\O {\rm P}^2$ and ${\rm Fl}(\O)$ carry natural structures of differentiable
manifolds.  More precisely, we have the natural identifications
\begin{equation}\label{flo}\O {\rm P}^2 = {\rm F}_4/{\rm Spin}(9) \quad {\rm and} \quad {\rm Fl}(\O) = {\rm F}_4/{\rm Spin}(8),\end{equation}
where ${\rm F}_4$ denotes the compact, connected, simply connected Lie group
whose Lie algebra is the (compact) real form of the complex simple
Lie algebra of type ${\rm F}_4$. 
We consider the natural $\O {\rm P}^1$-bundles
$\pi_1, \pi_2:{\rm Fl}(\O)\to \O {\rm P}^2$ given by
$$\pi_1(p,\ell) := p \quad {\rm and} \quad \pi_2(p,\ell) :=\ell.$$ 
Let also $\E_1$ and $\E_2$ denote the rank 8 vector bundles on
${\rm Fl}(\O)$ given by
\begin{equation}\label{e12}\E_1{(p,\ell)} := T_{(p,\ell)}\pi_1^{-1}(p)  \quad {\rm and} \quad
\E_2{(p,\ell)} := T_{(p,\ell)}\pi_2^{-1}(\ell).
\end{equation}
 We will use their Euler classes, $e(\E_1)$ and $e(\E_2)$, relative to appropriate orientations. They  both live in the integral cohomology ring of ${\rm Fl}(\O)$. 
Our first result gives a presentation of this ring in terms of generators and relations. Before stating it, we make the following convention, which will be used throughout  this paper: if it is 
not specified, the coefficient ring of a 
cohomology group is $\bR$.

\begin{theorem}\label{mainth} (a) The group $H^*({\rm Fl}(\O); \bZ)$ is free.

(b) 
We can orient the bundles $\E_1$ and $\E_2$  in such a way that the cohomology classes
$2e(\E_1)+e(\E_2)$ and $e(\E_1)+2e(\E_2)$ are multiples of 3.
Moreover,
the ring $H^*({\rm Fl}(\O);\bZ)$ is generated by 
$\frac{1}{3}(2e(\E_1)+e(\E_2))$ and $\frac{1}{3}(e(\E_1)+2e(\E_2))$,
the ideal of relations being generated by
$$S_i\left(\frac{1}{3}(2e(\E_1)+e(\E_2)),\frac{1}{3}(-e(\E_1)+e(\E_2)), 
-\frac{1}{3}(e(\E_1)+2e(\E_2))\right)=0$$
$i=2,3$.  Here  $S_2$ and $S_3$ denote the second, respectively third 
elementary symmetric polynomials in three variables.
\end{theorem}

The proof of this theorem is given in Section \ref{three}. It relies on  results of Hsiang, Palais, and Terng,
see \cite{Hs-Pa-Te}, concerning the rational
cohomology ring of  isotropy orbits of Riemannian symmetric spaces.

We also study the topology of ${\rm Fl}(\O)$ from the point of view of
the action of the group
$$M:={\rm Spin}(8)$$
which is  canonically induced by Equation  (\ref{flo}). 
More precisely, we are interested in the equivariant cohomology ring
$H^*_M({\rm Fl}(\O))$. 
We recall that this ring has a natural structure of  $H^*(BM)$-module,
which is defined as follows: we pick a point $x_0\in {\rm Fl}(\O)$ which is fixed by the
$M$-action  and consider the ring homomorphism 
$P^*:H^*(BM)=H^*_M(\{x_0\}) \to H^*_M({\rm Fl}(\O))$ induced by
the constant map $P : {\rm Fl}(\O)\to \{x_0\}$. 
We  define
$$a.\alpha:=P^*(a)\alpha,$$
for all $a\in H^*(BM)$ and all $\alpha\in H^*_M({\rm Fl}(\O))$.
In fact, $H^*_M({\rm Fl}(\O))$ becomes in this way a $H^*(BM)$-algebra.
It is a unital algebra, and this  provides  us with a canonical embedding of
$H^*(BM)$ into $H^*_M({\rm Fl}(\O))$; otherwise expressed, we identify 
$H^*(BM)$ with its image under $P^*$.   
As a general observation,  the fact that the $M$-equivariant group of a space is 
an $H^*(BM)$-algebra {\it with  unit} will be sometimes used in what follows without being explicitly mentioned.

It  is worth noting that, since
$M$ is a compact Lie group of rank four, $H^*(BM)$ is a polynomial
ring  with four generators. More precisely, we have
$$H^*(BM)=H^*(BT)^{W_M},$$
where $T\subset M$ is a maximal torus and $W_M$ the Weyl group of the pair $(M, T)$.
This gives
\begin{equation}\label{hgb}H^*(BM)=\bR[a_1,a_2,a_3,a_4],\end{equation}
where  $a_1$ lives in $H^4(BM)$, $a_2$ and $a_3$  in $H^8(BM)$, and $a_4$ in $H^{12}(BM)$ (see \cite[Section 3.7]{Hu}).  
The group $H^*(BM;\bZ)$ is described in 
\cite{Kon}; as it turns out from that description,  it contains 2-torsion elements,
and this is the reason which prevented us from
discussing  the $M$-equivariant cohomology with integer 
coefficients in this paper.

We will give two descriptions of the equivariant cohomology  ring 
$H^*_M({\rm Fl}(\O))$. We first note that the vector bundles $\E_1$ and $\E_2$ are $M$-equivariant and orientable, so we can
associate to them the equivariant Euler classes 
$e_M(\E_1)$ and $e_M(\E_2)$, which are elements of 
$H^8_M({\rm Fl}(\O))$. 
We also consider the equivariant Euler classes \begin{equation}
\label{eulercl}b_k:=e_M(\E_k|_{x_0}).\end{equation} $k=1,2$,
These two elements of 
$H^8_M(\{x_0\})=H^8(BM)$ are linearly independent and we have
$H^*(BM)=\bR[a_1, b_1, b_2, a_4]$ (see Lemma \ref{lasr}). 

\begin{theorem}\label{main}
We can orient the bundles $\E_1$ and $\E_2$  in such a way that,
as an $H^*(BM)$-algebra, $H^*_M({\rm Fl}(\O))$ is generated
by $e_M(\E_1)$ and $e_M(\E_2)$,
the ideal of relations being generated by:
\begin{align}
{}&S_i(2e_M(\E_1)+e_M(\E_2),-e_M(\E_1)+e_M(\E_2), -(e_M(\E_1)+2e_M(\E_2)))\label{ege1}\\
\nonumber {}&  \ \ \ \ \  \ \ \ \ \ \ \ \ \ \  \ \ \ \ \ \ \ \ \ \  \ \ \ \ \ \ \ \ \ \  \ \ \ \ \ \ \ \ \ \  \ \ \ \ \ \ \ \ \ \ =S_i(2b_1+b_2,-b_1+b_2,-(b_1+2b_2)),
\end{align}
$i=2,3$.
As before, $S_2$ and $S_3$ are the  
elementary symmetric polynomials of degree two, respectively three, in three variables.
\end{theorem}

The second result about $H^*_M({\rm Fl}(\O))$ gives a Goresky-Kottwitz-MacPherson
type presentation of this ring (cf.~\cite{Go-Ko-Ma}, where formulae for the 
equivariant cohomology of certain spaces with actions of tori have been obtained). We will see that the fixed point set
of the $M$-action on ${\rm Fl}(\O)$ can be identified with the symmetric group
$\Sigma_3$.   We put
\begin{equation}\label{conse}\tb_{1}:=b_1, \quad \tb_{2}:=-b_2, \quad \tb_{3}:=b_1-b_2.\end{equation}
We will show  the following:

\begin{theorem}\label{princ} 
(a) The (restriction) map 
$$\imath^*:H^*_M({\rm Fl}(\O))\to H^*_M(\Sigma_3)
=\prod_{\sigma\in \Sigma_3}H^*(BM) $$
induced by the inclusion
$\imath: \Sigma_3={\rm Fl}(\O)^M\hookrightarrow {\rm Fl}(\O)$
is injective. 

(b) The image of
 $\imath^*$ consists of all ordered sets
$(f_\sigma)\in \prod_{\sigma\in \Sigma_3}H^*(BM)$ such that 
$f_\sigma-f_{(i,j)\sigma}$ is a multiple of $\tb_i-\tb_j$,
for all $\sigma\in \Sigma_3$ and all 
$i ,j$ with $1\le i<j\le 3$. Here $(1,2)$, $(2,3)$, and $(1,3)$ denote the obvious elements (transpositions) of
$\Sigma_3$.  
\end{theorem}

This is a precise  description of $H^*_M({\rm Fl}(\O))$, if we take into account that
$$H^*(BM)=\bR[a_1, \tb_1, \tb_2, a_4],$$
see Lemma \ref{lasr}.

The last two theorems are proved in Sections \ref{four} and \ref{five} respectively.
Theorem \ref{princ} relies on a cell decomposition of ${\rm Fl}(\O)$, which is the analogue
of the classical Bruhat decomposition for complex flag manifolds; once we have this,
we simply  apply a result of Harada, Henriques, and Holm, see \cite{Ha-He-Ho}.
The proof of Theorem \ref{main} can be outlined as follows:
first, $e_M(\E_1)$ and $e_M(\E_2)$ generate $H^*_M({\rm Fl}(\O))$ as an
$H^*(BM)$-algebra, roughly because $H^*_M({\rm Fl}(\O))$ is isomorphic to
$H^*({\rm Fl}(\O))\otimes H^*(BM)$ as an $H^*(BM)$-module and
$H^*({\rm Fl}(\O))$ is generated as a ring by $e(\E_1)$ and $e(\E_2)$, see Theorem \ref{mainth};
secondly, one shows that if $f$ is any polynomial in three variables with coefficients
in $H^*(BM)$, then the restriction of the cohomology class
$f(2e_M(\E_1)+e_M(\E_2),-e_M(\E_1)+e_M(\E_2), -(e_M(\E_1)+2e_M(\E_2)))$ to 
an arbitrary $M$-fixed point $\sigma \in \Sigma_3$ is equal to
$g( 2b_1+b_2,-b_1+b_2,-(b_1+2b_2))$, where $g$ is a polynomial obtained from
$f$ by permuting the three variables in a certain way: this, along with the injectivity of
$\imath^*$, explains the relations (\ref{ege1}).

The last main result of the paper concerns the $M$-equivariant
$K$-theory ring of ${\rm Fl}(\O)$.  By the ``equivariant $K$-theory ring" of an 
$M$-space we always  mean
  the Grothendieck group of all topological 
$M$-equivariant complex vector bundles over that space,  with the multiplication  induced by the tensor
product (for more details, we refer  to  \cite{Se}).
To describe this ring  for ${\rm Fl}(\O)$, we need some information about the (complex) representation
ring $R[M]$ of $M$. 
It is known (see for instance \cite{Ad}) that the ring $R[{\rm Spin}(8)]$ is the
polynomial ring generated over $\bZ$ by $X_1, X_2, X_3, X_4$, which are
as follows: the  canonical representation 
of ${\rm SO}(8)$ on $\bC^8$ composed with the covering ${\rm Spin}(8)\to {\rm SO}(8)$,
the two complex half-{spin} representations, 
and the complexified adjoint action. Our result is a Goresky-Kottwitz-MacPherson type description of the ring
$K_M({\rm Fl}(\O))$.       

\begin{theorem}\label{kth}
The canonical 
homomorphism
$$K_M({\rm Fl}(\O))\to K_M(\Sigma_3)
=\prod_{\sigma\in \Sigma_3}R[M]= \prod_{\sigma\in \Sigma_3}\bZ[X_1,X_2,X_3,X_4]$$
induced by the inclusion
$\Sigma_3={\rm Fl}(\O)^M\hookrightarrow {\rm Fl}(\O)$
is injective. 
Its image consists of all
$ (f_\sigma) \in \prod_{\sigma \in \Sigma_3} \mathbb{Z} [X_1, X_2, X_3, X_4 ] $  such that 
$f_\sigma-f_{(i,j)\sigma}$ is a multiple of $X_i-X_j$,
for all $\sigma\in \Sigma_3$ and all 
$i ,j$ with $1\le i<j\le 3$.
Here $(1,2)$, $(2,3)$, and $(1,3)$ have the same meaning as in Theorem \ref{princ}.
\end{theorem}

A proof can be found in Subsection \ref{lastsec}. The main tool is again the theorem of Harada, Henriques, and Holm mentioned above.

\begin{rem} The omnipresence of the symmetric group $\Sigma_3$  in the above descriptions is not
surprising if we take into accout that ${\rm Fl}(\O)$ is diffeomorphic to the homogeneous
space ${\rm F}_4/{\rm Spin}(8)$.
It is well known that many geometric properties of ${\rm Spin}(8)$ and
${\rm F}_4$ involve $\Sigma_3$-symmetry. The generic term one uses for such phenomena is ``triality",  see \cite{Ad0}, 
\cite[Chapters 5 and 14]{Ad},  \cite{Ba}, and the references therein. For example, a result
that goes back to \'E.~Cartan in the 1920s says that the group of outer automorphisms of ${\rm Spin}(8)$ is isomorphic to $\Sigma_3$ and it acts on the ${\rm Spin}(8)$-modules
$V_8$, $S_8^+$, and $S_8^-$ by permuting them.
These representations of ${\rm Spin}(8)$ are also important for us here, as follows.
First, they induce the complex representations $X_1$, $X_2$, and $X_3$ which appear in Theorem \ref{kth}.
(Interesting enough, $X_4$, which is the adjoint representation of ${\rm Spin}(8)$, has no relevance in Theorem \ref{kth} and can practically be neglected.)   In the same spirit, it will turn out that the elements
$\tb_i-\tb_j$ of $H^*(B{\rm Spin}(8))$ we are using in Theorem \ref{princ} are the ${\rm Spin}(8)$-equivariant
Euler classes of $V_8$, $S_8^+$, and $S_8^-$, regarded as ${\rm Spin}(8)$-equivariant vector bundles over a point and equipped with appropriate orientations. The vector bundles $\E_1$ and $\E_2$ in Theorems \ref{mainth} and \ref{main}
are induced by $V_8$ and $S_8^+$, respectively, in a way which is described in Section \ref{rfm}.
We will see there that in the same way, to $S_8^-$ corresponds a third vector bundle, 
${\mathcal E}_3$, which we can use in order to bring even more $\Sigma_3$-symmetry 
into our first two theorems: this is explained in detail in Remarks \ref{tria1} and \ref{tria2}.       
\end{rem}  

\medskip

\begin{rem}\label{gener} The space ${\rm Fl}(\O)$ is a generalized real flag manifolds. By definition, such a manifold is an 
 orbit of the isotropy representation of a Riemannian symmetric space
 (for more details, see Appendix \ref{lasts1}).
The cohomology ring of the principal orbits of these representations was computed by Hsiang, Palais, and Terng in \cite{Hs-Pa-Te}. An important class of such manifolds consists of those
with uniform multiplicity 2, 4, or 8: these  are the principal adjoint orbits of compact Lie
groups,  the quaternionic flag manifold ${\rm Fl}_n(\H)$, and ${\rm Fl}(\O)$ respectively.
 The descriptions given in \cite{Hs-Pa-Te} show that the cohomology ring of each of these spaces is expressed by a Borel type formula, that is, it is isomorphic to the coinvariant ring of a certain Weyl group,
 see \cite{Bo0}.
The  spaces ${\rm Fl}_n(\H)$ and ${\rm Fl}(\O)$ admit natural group actions similar to the 
action of a maximal torus on an adjoint orbit (e.g., for ${\rm Fl}(\O)$ this group is
${\rm Spin}(8)$, see above). The equivariant cohomology and equivariant $K$-theory
of a principal adjoint orbit with the action of a maximal torus is well understood (see for example  \cite{Kos-Ku}). 
A natural goal is to decide whether ${\rm Fl}_n(\H)$ and ${\rm Fl}(\O)$ 
behave like adjoint orbits also in the equivariant setting.
Positive answers have been given for ${\rm Fl}_n(\H)$ 
from the point of view of equivariant cohomology  (see \cite{Ma2}) and  equivariant
$K$-theory (see \cite{Ma-Wi}). In this paper we discuss the remaining space,
which is  ${\rm Fl}(\O)$.    
\end{rem}

\medskip

\noindent{\bf Acknowledgement.} We would like to thank the referee for several valuable suggestions.

\newpage

\section{The octonionic flag manifold}\label{two}

The goal of this section is to define the flag manifold ${\rm Fl}(\O)$ 
and discuss some of its basic properties.
For reader's convenience we have included an appendix (see Appendix
\ref{lasts}) where the complex flag manifold ${\rm Fl}_3(\bC)$ is
discussed in a way appropriate to serve us as a model here.  

\subsection{${\rm Fl}(\O)$ via the Jordan algebra $(\h_3(\O),\circ)$}
We first recall  that, by definition, the space $\O$
has a basis consisting of the elements $e_1=1$, $e_2, \ldots, e_8$;
they satisfy certain multiplication rules which make $\O$ into
a non-associative algebra with  division (for more details, see \cite[Section 2]{Ba}).
Let $$p=x_1+x_2e_2+\ldots +x_8e_8$$ be an element of $\O$,
where $x_1,x_2,\ldots, x_8\in \bR$.
We define its real part,
$${\rm Re}(p):=x_1,$$
 its conjugate,
$$\overline{p}:=x_1-x_2e_2-\ldots -x_8e_8,$$
as well as its norm, $|p|$, given by
$$|p|^2:=p\cdot \overline{p}=x_1^2+x_2^2+\ldots +x_8^2.$$
Let us consider  
$$\h_3(\O):=
\left\{ \left(
\begin{array}{lll} x_1 & p & q\\
\bar{p} & x_2 & r\\
\bar{q}&\bar{r}&x_3
\end{array}
\right)
 \mid 
 p,q,r\in \O, \ x_1,x_2,x_3\in \bR
\right\},$$
the space of  all $3\times 3$ Hermitian matrices with entries in 
$\O$. 

\begin{definition}\label{defin}
(a) The {\rm octonionic projective plane }
$\O {\rm P}^2$ is the set of all matrices
$a\in \h_3(\O)$ with
$$a^2=a \ {\it and } \ {\rm tr}(a)=1.$$

(b) The {\rm octonionic flag manifold} 
${\rm Fl}(\O)$ is the set of all pairs $(a,b)\in \O {\rm P}^2\times \O {\rm P}^2$
with
$${\rm Re}({\rm tr}(ab))=0.$$
In the language of incidence geometry, this condition says
that the ``point" $a$ and the ``line" $b$ are ``incident" (see for instance
\cite[Section 7.2]{Fr}).  
\end{definition}
 
We equip $\h_3(\O)$  with the $\bR$-linear product\footnote{The pair
$(\h_3(\O),\circ)$ is actually a  Jordan algebra (see \cite{Ba} and \cite{Fr}).} given by
\begin{equation}\label{proda}a\circ b :=\frac{1}{2}(ab+ba),\end{equation}
for all $a,b\in \h_3(\O)$. 

\begin{definition}\label{gpf} The group ${\rm F}_4$ consists of all 
$\bR$-linear transformations $g$ of $\h_3(\O)$ 
such that
$$g.(a\circ b) =(g.a)\circ (g.b),$$
for all $a,b\in \h_3(\O)$.
\end{definition}

The following is a list of properties of the group ${\rm F}_4$ which will be needed
later. The details can be found for instance
in \cite{Fr}, \cite{Mu}, and \cite{Ad}.

\begin{itemize}
\item The group ${\rm F}_4$ is a compact, connected, simply connected Lie group
whose Lie algebra is the compact real form of the complex simple
Lie algebra of type ${\rm F}_4$.
\item  For any $a\in \h_3(\O)$ there exist $g\in {\rm F}_4$ and $x_1,x_2,x_3 \in \bR$
such that  
$$x_1\ge x_2 \ge x_3$$
and 
$$g.a
=
\left(
\begin{array}{llll}
x_1& 0 & 0\\
0 & x_2 &0\\
0&0&x_3
\end{array}
\right).$$
The numbers $x_1,x_2,x_3$ are uniquely determined by
$a$.
\item
We have
\begin{equation}\label{traces}{\rm tr}(g.a)={\rm tr} (a),\end{equation}
for all $g\in {\rm F}_4$ and all $a\in \h_3(\O)$.
\item We have
\begin{equation}\label{gii}g.I=I,\end{equation}
for any $g\in {\rm F}_4$. Here
$I$ denotes the diagonal matrix ${\rm Diag}(1,1,1)$.
\item Denote by $\d\cong\bR^3$ the space of all diagonal matrices in $\h_3(\O)$.  
We have 
\begin{equation}\label{fix}
\{g\in {\rm F}_4 \mid g.x = x \ {\rm for \ all \ } x\in \d\}\cong {\rm Spin}(8)
\end{equation}

\item
The space $\O {\rm P}^2$ is the ${\rm F}_4$-orbit of
$$d_1:=
\left(
\begin{array}{llll}
1& 0 & 0\\
0 & 0 &0\\
0&0&0
\end{array}
\right).$$
The stabilizer of $d_1$ is isomorphic to
the Lie group ${\rm Spin}(9)$.
Thus, we have the identification
$$\O {\rm P}^2 = {\rm F}_4/{\rm Spin}(9).$$ 
\end{itemize}  

We also have the following description of ${\rm Fl}(\O)$.

\begin{propo}\label{lal}  The (diagonal) action of ${\rm F}_4$ on 
${\rm Fl}(\O)$ is transitive.
If
$$d_2:=
\left(
\begin{array}{llll}
0& 0 & 0\\
0 & 1 &0\\
0&0&0
\end{array}
\right) 
$$
then the stabilizer of  $(d_1,d_2)$   is isomorphic to the group ${\rm Spin}(8)$
given by Equation (\ref{fix}). 
Thus, we have the identification
$${\rm Fl}( \O)
={\rm F}_4/{\rm Spin}(8).$$
\end{propo}
\begin{proof}
The transitivity of the ${\rm F}_4$-action follows from \cite[Sections 7.2 and 7.6]{Fr}.
The second assertion follows from the fact that $g\in {\rm F}_4$ fixes $\d$ pointwise
if and only if it fixes $d_1$ and $d_2$ (by Equation (\ref{gii})).
\end{proof}

Let us now consider the maps $\pi_1,\pi_2 : {\rm Fl}(\O)\to \O {\rm P}^2$, given by
$$\pi_1(a,b):=a \quad {\rm and } \quad \pi_2(a,b):=b,$$
for all $(a,b)\in {\rm Fl}(\O)$. From the previous considerations we deduce that
they are both ${\rm F}_4$-equivariant maps.

 \begin{propo}\label{flot} 
 The maps $\pi_1$ and $\pi_2$ 
are $\O {\rm P}^1$-bundles. Here, in analogy with
Definition \ref{defin} (a), $\O {\rm P}^1$ (the octonionic projective line) 
is the space of all 
idempotent elements of $\h_2(\O)$ with trace equal to 1.
\end{propo}

\begin{proof} 
We show that $\pi_1$ is an $\O {\rm P}^1$-bundle.
Since $\pi_1$ is ${\rm F}_4$-equivariant, it is sufficient to prove that $\pi_1^{-1}(d_1)=\O {\rm P}^1$
(because then, for any $g\in {\rm F}_4$ we have
$\pi_1^{-1}(g.d_1)=g.\O {\rm P}^1$).
Indeed, the elements of $\pi_1^{-1}(d_1)$ are of the form 
$(d_1,a)$, where $a\in \O {\rm P}^2$ is such that
$${\rm tr} (ad_1)=0.$$
The last equation  and the fact that
$a^2=a$ imply that
$$
a=
\left(
\begin{array}{llll}
0& 0 & 0\\
0 & x_2 &r\\
0&\bar{r}&x_3
\end{array}
\right)
$$
for $x_2,x_3\in \bR$ and $r\in \O$. 
The set of all such $a$ with $a^2=a$ and ${\rm tr}(a)=1$
 is the subspace $\O {\rm P}^1$ of $\{0\}\times \h_2(\O)$
(the latter being canonically embedded in $\h_3(\O)$).
 This finishes the proof.
 \end{proof}

\subsection{${\rm Fl}(\O)$ as a real flag manifold}\label{rfm}
Let $\h_3^0(\O)$ be the space of
all elements of $\h_3(\O)$ with trace equal to 0. 
The representation of ${\rm F}_4$ on the space
$\h_3(\O)$ mentioned in the previous subsection leaves
$\h_3^0$ invariant, see (\ref{traces}).
The main point of this subsection is that the induced representation of ${\rm F}_4$ on $\h_3^0(\O)$ is just the isotropy
representation of the (non-compact) Riemannian symmetric space ${\rm E}_{6(-26)}/{\rm F}_4$.
Here ${\rm E}_{6(-26)}$ is a certain non-compact real simple Lie group  whose
Lie algebra $\eg_{6(-26)}$ is a real form of the simple complex Lie algebra of type
${\rm E}_6$ (see \cite[Table V, Section 6, Ch.~X]{He}).  
Appendix \ref{lasts2} contains more details about this. We extract from there the
relevant information, as follows.
We have the Cartan decomposition\footnote{This also explains the subscript $-26$
from $\eg_{6(-26)}$. It is the signature of the Killing form of this Lie algebra. 
This form is negative definite
on $\fg_4$ (of dimension 52) and positive definite on $\h_3^0(\O)$ (of dimension
26).}  
\begin{equation}\label{card}\eg_{6(-26)}=\fg_4 \oplus \h^0_3(\O)\end{equation}
where $\fg_4$ is the Lie algebra of ${\rm F}_4$ and $\h_3^0(\O)$ the space of
all elements of $\h_3(\O)$ with trace equal to 0. We denote by  $\d^0$
the space  of all elements of $\d$ with trace equal to 0.
It is a maximal abelian subspace of $\h_3^0(\O)$.  
Let us also consider the 
following subspaces of $\h_3(\O)$:
$$\h_{\gamma_1}:=
\left\{
\left(
\begin{array}{llll}
0 & 0 & 0\\
0 & 0 & r\\
0 & \bar{r} & 0
\end{array}
\right) \mid r\in \O
\right\},
$$  
$$\h_{\gamma_2}:=
\left\{
\left(
\begin{array}{llll}
0 & 0 & q\\
0 & 0 & 0\\
\bar{q} & 0 & 0
\end{array}
\right) \mid q\in \O
\right\},
$$  
and
$$\h_{\gamma_3}:=
\left\{
\left(
\begin{array}{llll}
0 & p & 0\\
\bar{p} & 0 & 0\\
0 & 0 & 0
\end{array}
\right) \mid p\in \O
\right\}.
$$  
 We have the obvious decomposition
$$\h^0_3(\O)=\d^0\oplus \h_{\gamma_1}\oplus \h_{\gamma_2}\oplus \h_{\gamma_3}.$$
The spaces $\h_{\gamma_k}$ are in fact root spaces, in the sense that we have
\begin{equation}\label{hb}\h_{\gamma_k}
=\{a\in \h_3(\O) \mid 
[x,[x,a]]=\gamma_k(x)^2a\  {\rm for \ all \ } x\in \d^0\},\end{equation}
 $k=1,2,3$. 
 Here the bracket $[ \ ,  \ ]$ is the usual commutator of
 matrices and 
 $\gamma_1,\gamma_2,\gamma_3:\d^0 \to \bR$ are described by
\begin{align}\label{gam}{}&\gamma_1(x_1,x_2,x_3):=x_3-x_2,\nonumber\\
{}& \gamma_2(x_1,x_2,x_3):=x_1-x_3,\\
{}& \gamma_3(x_1,x_2,x_3):=x_1-x_2,\nonumber
\end{align}
where $(x_1,x_2,x_3)$ stands for ${\rm Diag}(x_1,x_2,x_3)$,
for any $x_1,x_2,x_3\in \bR$ with $x_1+x_2+x_3=0$
(for more details concerning Equation (\ref{hb}), see Appendix \ref{lasts2}).
The elements of $\Phi:=\{\pm \gamma_1,\pm\gamma_2,\pm\gamma_3\}$
are the roots\footnote{Strictly speaking, the roots are
$\pm\frac{1}{2}( x_3-x_2), \pm \frac{1}{2}( x_1-x_3) $,
and $\pm\frac{1}{2}( x_1-x_2)$ (see the end of Appendix \ref{lasts2}).}
of ${\rm E}_{6(-26)}/{\rm F}_4$ with respect to $\d^0$. We also consider the
subsets 
$$\Phi^+=\{ \gamma_1,\gamma_2,\gamma_3\} \quad {\rm and } \quad 
\Phi^-=\{- \gamma_1,-\gamma_2,-\gamma_3\}$$
of $\Phi$. They are the positive, respectively negative roots relative to the simple roots
$\gamma_1$ and $\gamma_2$. 
The following proposition concerns the action of ${\rm F}_4$ on $\h_3^0(\O)$ mentioned above.

\begin{propo}\label{x0123}
Take $x_0={\rm Diag}(x_1^0,x_2^0,x_3^0)\in \d^0$ 
such that  $x_1^0,x_2^0,$ and $x_3^0$ are any two different.
 Then
the ${\rm F}_4$-stabilizer of $x_0$ is the group ${\rm Spin}(8)$ in Proposition \ref{lal}.
One identifies in this way
\begin{equation}\label{fii}{\rm Fl}(\O) = {\rm F}_4.x_0.\end{equation}
\end{propo}

\begin{proof} An element $g\in {\rm F}_4$ leaves $x_0$ fixed if and only if
it leaves the entire $\d^0$ pointwise fixed (see Proposition \ref{x00}).   
By Equation (\ref{gii}) this is the same as saying that $g$ leaves 
$\d$ pointwise fixed. By Equation (\ref{fix}), this is equivalent to
$g\in {\rm Spin}(8)$.
\end{proof}

Consequently  ${\rm Fl}(\O)$ is a real {\rm fl}ag manifold (see Appendix \ref{lasts1} for more on this  notion). 
We deduce from this that the root spaces $\h_{\gamma_1},
 \h_{\gamma_2},$ and $\h_{\gamma_3}$ are ${\rm Spin}(8)$-invariant.
 In fact, the
 corresponding representations can be described explicitly as follows (see [Ba, p.~179]):
  \begin{itemize}
 \item $\h_{\gamma_1} = V_8$, the standard  (matrix) representation of ${\rm SO}(8)$
 on $\bR^8$, composed with the covering map $\pi: {\rm Spin}(8)\to {\rm SO}(8)$
 \item $\h_{\gamma_2} = S_8^+$
 \item $\h_{\gamma_3} = S_8^-$,
 \end{itemize}
 where $S_8^{\pm}$ are the two real half-{spin} 
 representations of ${\rm Spin}(8)$.

The Weyl group of ${\rm E}_{6(-26)}/{\rm F}_4$ with respect to $\d^0$ is 
\begin{equation}\label{weylg}W:=\{n\in {\rm F}_4 \mid n.\d^0\subset \d^0\}/{\rm Spin}(8).\end{equation}
The obvious action of this group  on $\d^0$ is faithful.
The corresponding group of transformations of $\d_0$ 
is generated by the reflections of $\d^0=\{(x_1,x_2,x_3)\in \bR^3 \mid 
x_1+x_2+x_3=0\}$ about the lines $\ker\gamma_1,
\ker \gamma_2$, and $\ker \gamma_3$ respectively. Thus, $W$ 
can be identified  with the symmetric group $\Sigma_3$
which acts   on  $\d^0$ by  permuting  the coordinates $x_1,x_2,x_3$.
Consequently, it also acts on $\Phi$, by
$$(\sigma \gamma)(x)=\gamma(\sigma^{-1}x),$$
for all $\sigma\in\Sigma_3$, $\gamma \in \Phi$, and $x\in \d^0$.

The tangent space to ${\rm Fl}(\O)$ (regarded as a submanifold of
 euclidean space $\h^0_3(\O)$) at the point $x_0$ introduced in Proposition \ref{x0123} is
$$T_{x_0}{\rm Fl}(\O)
=
\h_{\gamma_1}\oplus
 \h_{\gamma_2}\oplus\h_{\gamma_3}.$$
Consider the vector bundles
 $\E_1,\E_2,$ and  $\E_3$ on
${\rm Fl}(\O)$ given by
\begin{equation}\label{ek}\E_k|_{g.x_0}=g.\h_{\gamma_k},\end{equation}
for any $g\in {\rm F}_4$, $k=1,2,3$.
These are  sub-bundles of the tangent bundle of ${\rm Fl}(\O)$.  
In what follows we will show that $\E_1$ and $\E_2$ defined by Equation (\ref{ek}) are the same
as $\E_1$ and $\E_2$ defined by Equation (\ref{e12}).

\begin{propo}\label{foral}
The vector bundles $\E_1$ and $\E_2$ defined by Equation (\ref{ek}) satisfy
$$\E_1|_{g.x_0}=T_{g.x_0}\pi_1^{-1}(\pi_1(g.x_0))
\ {\it and } \ 
\E_2|_{g.x_0}=T_{g.x_0}\pi_2^{-1}(\pi_2(g.x_0))
$$
for all $g\in {\rm F}_4$.
\end{propo}

\begin{proof} 
 We prove the first equality.
By ${\rm F}_4$-equivariance, we only need to prove that
$$\h_{\gamma_1}=T_{(d_1,d_2)}\pi_1^{-1}(d_1).$$
Here we have used that $x_0$ corresponds to $(d_1,d_2)$ via the
isomorphism (\ref{fii}). 
We saw in the proof of Proposition \ref{flot}  that 
$\pi_1^{-1}(d_1)$ consists of all
$$
a=
\left(
\begin{array}{ccccc}
0& 0 & 0\\
0 & 1-x_3 &r\\
0&\bar{r}&x_3
\end{array}
\right)
$$
where $x_3\in \bR$ and $r\in \O$ such that 
$a^2=a$. This gives
$$|r|^2+\left(x_3-\frac{1}{2}\right)^2=\frac{1}{4}.$$
This is a full sphere in the 9-dimensional space $\O \times \bR$
 whose tangent space at
$(r,x_3)=(0,1)$ is described by $x_3=0$.
Regarded as a subspace of $\h_3^0(\O)$, the latter space is just $\h_{\gamma_1}$.
\end{proof}

Since ${\rm Fl}(\O)$ is a real flag manifold, we deduce from Appendix \ref{lasts1} (especially
Theorem \ref{dkv}) that it
 has the following natural cell decomposition:
\begin{equation}\label{cwdec}{\rm Fl}(\O)=\bigsqcup_{\sigma\in \Sigma_3} C_\sigma.\end{equation}
 For each $\sigma\in \Sigma_3$, the cell $C_\sigma$ is 
 invariant under the action of ${\rm Spin}(8)$ and  we have a  ${\rm Spin}(8)$-equivariant diffeomorphism
\begin{equation}\label{cs}C_\sigma\cong\bigoplus \h_\gamma\end{equation} 
 where the sum runs over all $\gamma\in \Phi^+$ such that 
 $\sigma^{-1}\gamma\in\Phi^-$ (see Corollary \ref{cwd}). The following result will play an important role
 in our investigation:
 
 \begin{propo}\label{eac} Each $C_\sigma$ can be identified with $\bC^{n(\sigma)}$ for some number 
$n(\sigma)$. In this way, the canonical maximal torus
$T$ of ${\rm Spin}(8)$ (see for example \cite[Ch.~3]{Ad} or
\cite[Ch.~IV, Theorem 3.9]{Br-tD}) acts $\bC$-linearly on  $C_\sigma$. 
\end{propo}

\begin{proof} By the decomposition (\ref{cs}), it is sufficient to study the action of $T$ on
$V_8$, $S_8^+$, and $S_8^-$. The last two representations of
${\rm Spin}(8)$ are obtained from the first one by (outer) automorphisms 
of ${\rm Spin}(8)$ (see \cite[Theorem 5.6]{Ad}).
Since any of these automorphisms leave $T$ invariant, 
it is sufficient to consider the action of
$T$ on $V_8$. 
Without giving the exact description of $T$ (see the references  above), 
we recall that if
$\pi: {\rm Spin}(8)\to {\rm SO}(8)$ is the canonical double covering, then
the elements of $\pi(T)$ are block diagonal $8\times 8$ matrices consisting of four  blocks of the form
$$\left(
\begin{array}{rrrrr}
\cos \theta & -\sin\theta\\
\sin\theta & \cos \theta
\end{array}
\right),
$$      
where $\theta\in \bR$. If we identify $\bR^8=\bC^4$ via
$$(x_1,x_2, \ldots, x_7,x_8)
=(x_1+ix_2,\ldots, x_7+ix_8),$$
then the action of any element of $T$ is given by four copies of a  map
of the form $$x_1+ix_2\mapsto (\cos \theta +i\sin\theta)(x_1+ix_2)$$
for all $x_1+ix_2\in \bC$. This map is obviously $\bC$-linear 
(since the multiplication of complex numbers is commutative). \end{proof}

Finally, we describe the fixed points of the 
${\rm Spin}(8)$-action on ${\rm Fl}(\O)$.

\begin{propo}\label{fixp}
The fixed point set of the ${\rm Spin}(8)$-action on
${\rm Fl}(\O)={\rm F}_4.x_0$ is
$${\rm Fl}(\O)^{{\rm Spin}(8)} = \Sigma_3x_0.$$
If $T\subset {\rm Spin}(8)$ is the canonical maximal torus, then the
fixed points of the $T$- and the ${\rm Spin}(8)$-action on
${\rm Fl}(\O)$ are the same.
\end{propo}

\begin{proof} 
We start with the following claim.

\noindent {\it Claim.} If $a\in \h_3^0(\O)$ is fixed 
 by $T$, then $a$ is in $\d^0$. 
 
 \noindent To prove this we decompose
 $$a=a_0+a_1+a_2+a_3,$$
where $a_0\in \d^0$ and $a_j\in \h_{\gamma_j}$, $j=1,2,3$.
Since $\d^0, \h_{\gamma_1},\h_{\gamma_2}$, and $\h_{\gamma_3}$ are ${\rm Spin}(8)$-invariant
(see above),
all four of $a_0,a_1,a_2,a_3$ are fixed by $T$.
Assume that $a$ is not in $\d^0$.
Then at least one of $a_1$, $a_2$, and $a_3$ is non-zero.
Say first that $a_1$ is non-zero. 
We have
\begin{equation}\label{pi}\pi(g)\cdot a_1=a_1,\end{equation}
for all $g\in T$.  Here $\pi:{\rm Spin}(8)\to {\rm SO}(8)$ is the canonical double
covering and ``$\cdot$" is the matrix multiplication. The ${\rm SO}(8)$-stabilizer
of $a_1$ is isomorphic to ${\rm SO}(7)$. Equation (\ref{pi}) says that
this stabilizer contains the four dimensional torus $\pi(T)$ as a subgroup, which
 contradicts ${\rm rank}({\rm SO}(7))=3$.  
 If $a_2$ (or $a_3$) is different from 0, the argument we use is similar:
 the representation of ${\rm Spin}(8)$ on $\h_{\gamma_2}=S_8^+$
 (respectively on $\h_{\gamma_3}=S_8^-$)
 differs from $V_8$ by an
 (outer) automorphism of ${\rm Spin}(8)$.
 
The claim implies that
 $${\rm Fl}(\O)^T\subset {\rm Fl}(\O)\cap \d^0=\Sigma_3x_0.$$
For the last equality we have used  \cite[Section 5 (Hauptachsentransformation von ${\mathfrak I}$)]{Fr} (see also
\cite[Section 5, Lemma 1]{Mu}).
 On the other hand, Equation (\ref{fix}) implies that
 $${\rm Fl}(\O)\cap \d^0 \subset {\rm Fl}(\O)^{{\rm Spin}(8)}.$$ 
 This finishes the proof.
\end{proof}

\section{Cohomology of ${\rm Fl}(\O)$}\label{three}

Let us consider again the projection maps $\pi_1,\pi_2 : {\rm Fl}(\O)\to \O {\rm P}^2$  defined by Equation (\ref{e12}). We would like to describe $\pi_1$ and $\pi_2$ by using the identification between
${\rm Fl}(\O)$ and the orbit ${\rm F}_4.x_0$ (see Proposition \ref{x0123}). 
To this end, we consider the following two elements of $\d^0$:
$$
d_1^0:=d_1-\frac{1}{3}I=
\left(
\begin{array}{ccccc}
\frac{2}{3}& 0 & 0\\
0 & -\frac{1}{3} &0\\
0&0&-\frac{1}{3}
\end{array}
\right)
\quad {\rm and} \quad
d_2^0:=d_2-\frac{1}{3}I=
\left(
\begin{array}{ccccc}
-\frac{1}{3}& 0 & 0\\
0 & \frac{2}{3} &0\\
0&0& -\frac{1}{3}
\end{array}
\right).
$$
For each of them, the ${\rm F}_4$-stabilizer is a copy of ${\rm Spin}(9)$ which
contains the ${\rm F}_4$-stabilizer of $x_0$, see Proposition \ref{x0123}. Thus, the ${\rm F}_4$-orbits
of $d_1^0$ and $d_2^0$  are both diffeomorphic 
to $\O {\rm P}^2$.
The maps
$$ p_1: {\rm F}_4.x_0 \to {\rm F}_4.d_1^0 \quad {\rm and} \quad p_2:{\rm F}_4.x_0\to {\rm F}_4.d_2^0$$
given by $p_1(g.x_0)=g.d_1^0$ and $p_2(g.x_0)=g.d_2^0$ are well defined.
Let us consider the following  diagram:
$$\xymatrix{
    {\rm Fl}(\O) \ar[rr]^{\pi_1}   \ar@{->}[dd]^{} & &
    \ar@{->}[dd]^{}  \O {\rm P}^2 \\
     \\
   \ar[rr]^{{p_1}} {\rm F}_4.x_0& & 
  {\rm F}_4.d_1^0} $$
Here, the vertical arrow in the left-hand side is the ${\rm F}_4$-equivariant diffeomorphism
which maps  $(d_1,d_2)$ to $x_0$ (see Proposition \ref{x0123}). The other vertical arrow in the diagram  is the diffeomorphism given by $$x\mapsto x-\frac{1}{3}I,$$
for all $x\in \O {\rm P}^2$: it is an ${\rm F}_4$-equivariant diffeomorphism too. The diagram is commutative. We also have a similar diagram
which involves $p_2$ and $\pi_2$.  
Thus, if  we   identify
$${\rm F}_4.x_0={\rm Fl}(\O), \ {\rm F}_4.d_1^0=\O {\rm P}^2, \ {\rm and} \  {\rm F}_4.d_2^0=\O {\rm P}^2  $$
then we have the following result:

\begin{proposition}\label{themapsp}
The maps $p_1, p_2: {\rm Fl}(\O)\to \O {\rm P}^2$ defined above are
${\rm Spin}(8)$-equivariant $\O {\rm P}^1$-bundles.
The vector bundles $\E_1$ and $\E_2$ defined by Equation (\ref{ek}) satisfy
$$\E_1|_{g.x_0}=T_{g.x_0}p_1^{-1}(p_1(g.x_0))
\ {\it and } \ 
\E_2|_{g.x_0}=T_{g.x_0}p_2^{-1}(p_2(g.x_0))
$$
for all $g\in {\rm F}_4$.
\end{proposition}

This proposition is a direct consequence of Propositions 
\ref{flot} and \ref{foral}.

We will use the notation
$$X:={\rm Fl}(\O)={\rm F}_4.x_0.$$ 

Let us consider again
the functions  $\gamma_1,\gamma_2,\gamma_3:\d^0\to \bR$ defined in the previous section (actually the restrictions to $\d^0$ of the functions
given by Equation (\ref{gam})). Recall that
$\{\pm \gamma_1,\pm\gamma_2,\pm\gamma_3\}$
is a root system of type $A_2$. 
We choose the simple root system consisting of  $\gamma_1$ and $\gamma_2$;
then $\gamma_3=\gamma_1+\gamma_2$ is the third positive root.

In what follows we will construct an orientation on each of the bundles $\E_k,$ $k=1,2,3$.
 First, we pick an orientation on $\E_k|_{x_0}=\h_{\gamma_k}$ (see below).
Then, if $g\in {\rm F}_4$, we choose the orientation on $\E_k|_{g.x_0}=g.\h_{\gamma_k}$
in such a way  that the map $g$ is orientation preserving (note that this definition
does not depend on $g$, since  the stabilizer group
$({\rm F}_4)_{x_0}={\rm Spin}(8)$ is connected and each of its elements acts on $\h_{\gamma_k}$ as a linear orthogonal
transformation, see Section 2).   
Thus, orienting $\E_1, \E_2$, and $\E_3$ amounts to choosing orientations on
$\h_{\gamma_1}, \h_{\gamma_2}$, and $\h_{\gamma_3}$.
We proceed as follows. 
First we take into account that $\gamma_3=s_2\gamma_1$, where
$s_2$ denotes the element of the Weyl group $W$ given by
the reflection of $\d^0$ about $\ker\gamma_2$ (see (\ref{weylg}) for the definition of $W$). There exists
$n_2\in {\rm F}_4$ with   $n_2.\d^0=\d^0$ such that $s_2$ is equal to the
coset $[n_2]=n_2{\rm Spin}(8)$ in $\Sigma_3$.
Consequently, we have
$$\gamma_3=\gamma_1\circ n_2^{-1}.$$
This implies that $n_2$ maps $\h_{\gamma_1}$ to $\h_{\gamma_3}$.   
Similarly, there exists $n_1\in {\rm F}_4$ such that 
$$\gamma_3=\gamma_2\circ n_1^{-1}.$$
Thus, $n_1^{-1}$ maps $\h_{\gamma_3}$ to $\h_{\gamma_2}$.
We pick and fix an orientation on $\h_{\gamma_1}$;
 the orientations we equip $\h_{\gamma_2}$ and $\h_{\gamma_3}$ with are
such that the maps $n_1$ and $n_2$ are orientation preserving.

The main goal of this section is to prove Theorem \ref{mainth}.
We proceed as follows.
First, observe that
$X$ is a 24-dimensional manifold. 
It is known (see for instance \cite[Section 5]{Hs-Pa-Te}) that 
the group $H^*(X;\bZ)$  is a free $\bZ$-module
such that 
\begin{equation}\label{dimh}{\rm rank} \ H^k(X;\bZ)=
\begin{cases}
0, \ {\rm if \ }  k\notin \{0,8,16, 24\}\\
2, \ {\rm if \ } k\in \{8,16\}\\
1, {\rm \ if \ } k\in \{0,24\}.
\end{cases}
\end{equation}
 A basis of $H^8(X;\bZ)$ can be constructed  as follows.
By Proposition \ref{themapsp}, the subspaces $$\S_1:=p_1^{-1}(d_1^0)
 \quad {\rm and} \quad \S_2:=p_2^{-1}(d_2^0)$$ of ${\rm Fl}(\O)$ are 
diffeomorphic to $\O {\rm P}^1$, hence to the sphere $S^8$. By Proposition \ref{themapsp},
the tangent bundle of $\S_1$ is just $\E_1|_{\S_1}$: thus, the orientation of $\E_1$ chosen above induces an orientation on $\S_1$. Similarly we can also orient $\S_2$. 
  The homology classes $[\S_1]$ and $[\S_2]$ carried by $\S_1$ and $\S_2$ 
 are a basis  of $H_8(X;\bZ)$.
  Thus, the cohomology classes $\beta_1,\beta_2\in H^8(X;\bZ)$ determined by
 $$ ( \beta_i, [\S_j]) =
 \begin{cases}
 1, \ {\rm if } \ i=j\\
 0,\ {\rm otherwise}
 \end{cases}
 ,$$
$1\le i,j\le 2$, are a basis of $H^8(X;\bZ)$
(here $( \ , \ ) :H^8(X;\bZ)\otimes
H_8(X;\bZ) \to \bZ$ denotes the evaluation pairing). 

We take into account that the elements $d_1^0$ and $d_2^0$ of $\d^0$
satisfy $\gamma_1(d_1^0)=0$ and
$\gamma_2(d_2^0)=0$. The following equations can be deduced from 
\cite[Proof of Theorem 6.12]{Hs-Pa-Te}
(see also 
\cite[Proof of Lemma 3.3]{Ma1}):
\begin{align*}{}& e(\E_1)= 2\beta_1+
\frac{2\langle \gamma_1,\gamma_2\rangle}{\langle \gamma_2,\gamma_2\rangle}
\beta_2
=2\beta_1-\beta_2\\
{}&e(\E_2)= \frac{2\langle \gamma_2,\gamma_1\rangle}{\langle \gamma_1,\gamma_1\rangle}
\beta_1 +2\beta_2
=-\beta_1+2\beta_2\\
{}& e(\E_3)=e(\E_1)+e(\E_2).
\end{align*}
Thus, we have
\begin{equation}\label{beta}\beta_1
=\frac{1}{3}(2e(\E_1)+e(\E_2)) \quad  {\rm and } \quad 
\beta_2
=\frac{1}{3}(e(\E_1)+2e(\E_2)).
\end{equation}

From Equation (\ref{ek}) we deduce that
the tangent bundle $TX$ can be split as
$$TX=\E_1\oplus \E_2\oplus \E_3.$$
This implies:
$$e(TX)=e(\E_1)e(\E_2)e(\E_3)
=e(\E_1)e(\E_2)(e(\E_1)+e(\E_2)).$$
If $[X]$ is the fundamental homology class of $X$, then
$$( e(\E_1)e(\E_2)(e(\E_1)+e(\E_2)),[X])
=(e(TX),[X])
=\chi(X)=6,$$
where $\chi(X)$ is the Euler-Poincar\'e characteristic of $X$.
We know it is equal to 6 by Equation (\ref{dimh}).
Consequently, the cohomology class
\begin{equation}\label{1over6}\frac{1}{6}e(\E_{1})e(\E_{2})(e(\E_{1})+e(\E_{2}))
\end{equation}
is a basis of $H^{24}(X;\bZ)$ over $\bZ$. 

Let us now consider separately the root system $\{\pm\gamma_1,\pm\gamma_2,\pm\gamma_3\}$.
The fundamental weights corresponding to the simple roots 
$\gamma_1,\gamma_2$ are
$$\lambda_1=\frac{1}{3}(2\gamma_1+\gamma_2) \quad  {\rm and } \quad 
\lambda_2=\frac{1}{3}(\gamma_1+2\gamma_2).$$
We know that there exists a canonical isomorphism
between the ring
$$\bQ[\lambda_1,\lambda_2]/\langle 
{\rm  nonconstant \ symmetric  \ polynomials  \ in \ }
\lambda_1, \lambda_2-\lambda_1,-\lambda_2\rangle$$
and $H^*({\rm Fl}_3(\bC);\bQ)$, see \cite{Bo0}.  By a theorem of
Bernstein, I.~M.~Gelfand, and S.~I.~Gelfand, see \cite{Be-Ge-Ge},  the Schubert basis of 
$H^*({\rm Fl}_3(\bC);\bQ)$ over $\bQ$ is obtained
by considering (the coset of)
$$\frac{1}{6}\gamma_1\gamma_2(\gamma_1+\gamma_2)$$
 and applying successively  the divided difference   operators
 $\Delta_{\gamma_1}$ and $\Delta_{\gamma_2}$.
 Here, the operator $\Delta_\gamma$ corresponding to the root
 $\gamma\in \{\gamma_1,\gamma_2\}$ is defined by
 $$\Delta_{\gamma}(f)=\frac{f-f\circ s_\gamma}{\gamma}$$
 for any $f\in \bQ[\lambda_1,\lambda_2]$ (by $s_\gamma$ we denote the reflection  about the line $\ker\gamma$).
 The Bernstein-Gelfand-Gelfand basis of $H^*({\rm Fl}_3(\bC);\bQ)$ mentioned
 above consists of  the cosets of the following polynomials:
 $$\frac{1}{3}\gamma_1(\gamma_1+\gamma_2), \ 
 \frac{1}{3}\gamma_2(\gamma_1+\gamma_2)$$
 $$\lambda_1, \ \lambda_2$$
 $$1.$$
The Schubert classes corresponding to $\lambda_1$ and $\frac{1}{3}\gamma_1(\gamma_1+\gamma_2)$
are Poincar\'e dual to each other, hence we have: 
\begin{equation}\lambda_1\cdot \label{frac}\frac{1}{3}\gamma_1(\gamma_1+\gamma_2)=\frac{1}{6}\gamma_1\gamma_2(\gamma_1+\gamma_2)+f\end{equation}
 where $f\in \bQ[\lambda_1,\lambda_2]$ is in the ideal generated by the non-constant symmetric polynomials
 in $\lambda_1,\lambda_2-\lambda_1,-\lambda_2$.
 
We now return to the cohomology of $X$. By
\cite[Theorem 6.12]{Hs-Pa-Te} (see also
\cite[Section 3]{Ma1}),  the ring $H^*(X;\bQ)$ is generated by $\beta_1$ and
$\beta_2$, the ideal of relations being generated by the
symmetric polynomials in $\beta_1,\beta_2-\beta_1,-\beta_2$.
Equations (\ref{beta}) and (\ref{frac}) imply that the equality
$$\beta_1\frac{1}{3}e(\E_{1})(e(\E_{1})
+e(\E_{2}))=
\frac{1}{6}e(\E_{1})e(\E_{2})(e(\E_{1})+e(\E_{2})),$$
holds in $H^*(X;\bQ)$. The right-hand side of the equation is the fundamental cohomology class of $X$ over $\bZ$ (see 
Equation (\ref{1over6})).  Since $\beta_1$ is in $H^*(X;\bZ)$, we deduce that
 the cohomology class
 \begin{equation}\label{13}\frac{1}{3}e(\E_{1})(e(\E_{1})
+e(\E_{2}))\end{equation}
belongs to $H^*(X;\bZ)$, being  the Poincar\'e dual of $\beta_1$ in $H^*(X;\bZ)$. Similarly, the class
\begin{equation}\label{12}\frac{1}{3}e(\E_{2})(e(\E_{1})
+e(\E_{2}))\end{equation}
is in $H^{*}(X;\bZ)$, being the Poincar\'e dual of $\beta_2$.
Consequently, the classes given by (\ref{13}) and (\ref{12}) are 
a basis of $H^{16}(X;\bZ)$.
 
 To complete the proof, it only remains to show that the cohomology classes
 given by (\ref{13}) and (\ref{12}) can be expressed as polynomials with
 integer coefficients in $\beta_1$ and $\beta_2$. 
 Indeed, by using (\ref{beta}), we can see that
 $$\frac{1}{3}e(\E_{1})(e(\E_{1})
+e(\E_{2}))=\beta_1^2$$
and
$$\frac{1}{3}e(\E_{2})(e(\E_{1})
+e(\E_{2}))=\beta_2^2.$$
Here we have used the relation
$$\beta_1^2+\beta_2^2-\beta_1\beta_2=0,$$
which follows from the fact that the second symmetric polynomial in $\beta_1,\beta_2-\beta_1,-\beta_2$ is
equal to 0. 

\begin{rem}\label{tria1} Denoting by ${\mathcal E}^-_3$ the vector bundle ${\mathcal E}_3$ 
equipped with the orientation which is opposite to the one we have defined above, we can rephrase Theorem \ref{mainth} by saying that
the ring $H^*({\rm Fl}(\O);\bZ)$ is generated by 
$$x_1:=\frac{1}{3}(e(\E^-_3)-e(\E_2)), \ x_2:=\frac{1}{3}(e(\E_1)-e(\E^-_3)),    \ x_3:=\frac{1}{3}(e(\E_2)-e(\E_1)), $$ subject
 to the relations given by the vanishing of the symmetric polynomials in $x_1, x_2$, and $x_3$.
 See also Appendix \ref{lasts} and Remark \ref{gener}. 
\end{rem}

\begin{rem} The result stated in Theorem \ref{mainth} is not
entirely new: a similar description has been obtained for example 
in \cite[Theorem 2.3]{Yo} (cf.~also \cite[Lemma 20.4]{Bo}). The novelty of Theorem \ref{mainth} is that it gives geometric
descriptions of the generators of the cohomology ring.
\end{rem}

\section{Equivariant cohomology of ${\rm Fl}(\O)$: generators and relations}\label{four}

In this section we will prove Theorem \ref{main}.
As before, we denote 
$$M:={\rm Spin}(8) \quad {\rm and } \quad X:={\rm Fl}_3(\O)={\rm F}_4.x_0.$$
We first recall that the vector spaces $\h_{\gamma_1}, \h_{\gamma_2}$, and $\h_{\gamma_3}$,
as well as  the vector bundles $\E_1, \E_2$, and $\E_3$ have been endowed with orientations in Section \ref{three}.  
The following result   is an immediate consequence of Theorem \ref{mainth} (see also Section \ref{three}):

\begin{lemma}\label{allc} The ring  $H^*(X)$ is generated by $e(\E_1)$ and $e(\E_2)$, subject to  the relations
\begin{align*}
 S_i(2e(\E_1)+e(\E_2),-e(\E_1)+e(\E_2), -e(\E_1)-2e(\E_2))=0,
\end{align*}
$i=2,3$. Here $S_i$ denotes the $i$-th fundamental symmetric polynomial in three variables.
\end{lemma}

We are actually interested here in the equivariant cohomology ring $H^*_M(X)$.
We recall that, by definition, we have $H^*_M(X)= H^*(EM \times_M X)$, where
$EM$ is the total space of the classifying principal bundle $EM\to BM$ of $M$.
As explained in the introduction, $H^*_M(X)$ has a canonical structure of $H^*(BM)$-module.
In the case at hand, this module turns out to be free, of rank equal to $\dim H^*(M)$: we say
that the $M$-action on $X$ is {\it equivariantly formal}. This follows readily from the fact that
$H^{\rm odd}(X)=\{0\}$, see Lemma \ref{allc},  and some standard results in equivariant cohomology, see for example
\cite[Lemma C.24 and Proposition C.26]{Gu-Gi-Ka}.
The result stated in the following proposition is a direct  
 consequence of equivariant
formality (see for example \cite[Proposition 4.4]{Ha-Ho}). 

\begin{proposition}\label{le}  The  graded ring homomorphism
$\jmath^*:H^*_M(X) \to H^*(X)$ induced by the canonical
inclusion $\jmath : X \to EM\times_M X$
is surjective. Its kernel is
$$\ker \jmath^* =\langle  H^+(BM).H^*_M(X)\rangle,$$
where $H^+(BM)$ denotes the space of
all elements of $H^*(BM)$ of strictly positive degree
and $\langle  H^+(BM).H^*_M(X)\rangle$ is the 
$\bR$-span of all elements of the form
$a.\alpha$, with $a\in H^+(BM)$ and
$\alpha \in H^*_M(X)$. 
\end{proposition}

Our first goal is to prove that Equation  (\ref{ege1}) hold true.
The elements $b_1,b_2$ of $H^*(BM)$ involved there can actually 
be expressed as 
 $$b_k=e_M(\h_{\gamma_k}),$$
$k=1,2$ (see Proposition \ref{foral}). Let us also define
\begin{equation}\label{letusdef}b_3:=e_M(\h_{\gamma_3}).\end{equation}
The following notation is standard: if $\alpha\in H^*_M(X)$ and $x\in X^M$,
then the  restriction of $\alpha$ to $x$ is $$\alpha|_x:=i_x^*(\alpha),$$
where $i_x:\{x\}\to X$ is the inclusion map (note that 
$\alpha|_x\in H^*_M(\{x\})=H^*(BM)$). 
The following lemma will be needed later. It is worthwhile  recalling at this point  that $H^*(BM)$ is identified via
$P^*$ with a subspace of $H^*_M(X)$, see Section \ref{secfirst}. 

\begin{lemma}
We have
\begin{equation}\label{ege}
e_M(\E_1)+e_M(\E_2)-e_M(\E_3)=b_1+b_2-b_3.\end{equation}
\end{lemma}
\begin{proof} For any $k\in\{1,2,3\}$ we have
$\jmath^*(e_M(\E_k))=e(\E_k)$ (since, by definition, $e_M(\E_k)$ is the Euler class
of a vector bundle over $EM\times_M X$ whose pullback via $\jmath$ is $\E_k$). 
We deduce that 
 $$\jmath^*(e_M(\E_1)+e_M(\E_2)-e_M(\E_3))=e(\E_1)+e(\E_2)-e(\E_3)=0.$$
 From Proposition \ref{le} and the fact that $H^k(X)=\{0\}$ for all
 $1\le k \le 7$ (see Equation (\ref{dimh})) we deduce that 
 $$e_M(\E_1)+e_M(\E_2)-e_M(\E_3)\in P^*(H^*(BM)).$$
The composition $P\circ i_{x_0}$ is the identity function of $\{x_0\}$.
Thus, it is now sufficient to note that
 $$(e_M(\E_1)+e_M(\E_2)-e_M(\E_3))|_{x_0}
 =e_M(\E_1|_{x_0})+e_M(\E_2|_{x_0})-e_M(\E_3|_{x_0})
 =b_1+b_2-b_3,$$
 where we have used Equation (\ref{ek}).
 \end{proof}
 
The following localization result will also be used here.
It will be proved in Subsection \ref{five}.
We recall, see Lemma \ref{fixp}, that 
the fixed points of  the $M$-action on
$X$ are given by
$$X^M=\Sigma_3x_0$$ 

\begin{lemma} 
The restriction  map 
$$H^*_M(X)\to H^*_M(X^M)$$
is injective.
\end{lemma}

The strategy we will use in order to justify  (\ref{ege1}) is by
showing for each equation that the two  sides are equal when restricted to
any point in $X^M$. We recall that $X^M$ is equal to the $W$-orbit of $x_0$,
where $W$ acts on $\d^0$ as the 
reflection group of the root system
$\{\pm \gamma_1, \pm\gamma_2, \pm\gamma_3\}.$
Since $\{\gamma_1,\gamma_2\}$ is a simple root system, the reflections
$s_1:=s_{\gamma_1}$ and $s_2:=s_{\gamma_2}$ generate $W$. Moreover,
we have \begin{equation}\label{wey=}W=\{1, s_1, s_2, s_1s_2, s_2s_1, s_1s_2s_1\},\end{equation} where $s_1s_2s_1=s_2s_1s_2$.

\begin{lemma}\label{ther} The restrictions of $e_M(\E_1)$, 
$e_M(\E_2)$, and $e_M(\E_3)$ to
$X^M$ are as follows:

\begin{tabular}{|l|rrrrrr|}
	\hline
$\sigma$& $1$ &  $s_1$    &   $s_2$  & $s_1s_2$   & $s_2s_1$& $s_1s_2s_1$  \\
	\hline
$e_M(\E_1)|_{\sigma x_0}$& $b_1$   & $-b_1$ &$b_3$ & $b_2$  &$-b_3$ & $-b_2$  \\
	\hline
$e_M(\E_2)|_{\sigma x_0}$& $b_2$   & $b_3$ &$-b_2$ & $-b_3$  &$b_1$ & $-b_1$  \\
	\hline
$e_M(\E_3)|_{\sigma x_0}$& $b_3$   & $b_2$ &$b_1$ & $-b_1$  &$-b_2$ & $-b_3$  \\
	\hline
\end{tabular}

\end{lemma}

\begin{proof} 
We have
$$e_M(\E_1)|_{s_1x_0}=e_M\left(\E_1|_{s_1x_0} \right).$$
By definition,  $\E_1|_{x_0}=\h_{\gamma_1}$.
The points $x_0$ and $s_1x_0$ are  antipodal points of the eight dimensional sphere $\S_1=p_1^{-1}(d_1^0)$,
which is embedded in $X$ (see Section \ref{three}).
By Proposition \ref{themapsp}, the tangent bundle of $\S_1$ is just the restriction of $\E_1$ to $\S_1$.
The orientation of $\E_1$ induces an orientation of the sphere $\S_1$.
The space $\E_1|_{s_1x_0}$
is the same as $\E_1|_{x_0}=\h_{\gamma_1}$, but with the reversed orientation. 
  Consequently,
$$e_M(\E_1)|_{s_1x_0}=-e_M(\h_{\gamma_1})=-b_1.$$

Let us now determine
$$e_M(\E_1)|_{s_2x_0}=e_M\left(\E_1|_{s_2x_0} \right).$$
Like in Section \ref{three}, we consider again $n_2\in {\rm F}_4$ such that $s_2$ is the coset
$[n_2]=n_2{\rm Spin}(8)$ in the Weyl group $W=\Sigma_3$.
By definition, since $s_2x_0=n_2.x_0$,
 we have
$$\E_1|_{s_2x_0}=n_2.\h_{\gamma_1}.$$
Moreover,  $n_2$ is an orientation preserving map from $\h_{\gamma_1}$
to $\E_1|_{s_2x_0}$ (from the way we have oriented $\E_1$ in Section \ref{three}). 
On the other hand, we saw in Section \ref{three} that $n_2$ 
maps $\h_{\gamma_1}$ to $\h_{\gamma_3}$ by preserving the orientation. We deduce that
$$e_M\left(\E_1|_{s_2x_0} \right)=e_M(\h_{\gamma_3})=b_3.$$

We determine now 
$$e_M(\E_1)|_{s_1s_2x_0}=e_M\left(\E_1|_{s_1s_2x_0} \right).$$
Take $n_1\in {\rm F}_4$ such that $s_1=[n_1]=n_1{\rm Spin}(8)$ in $\Sigma_3$.
We have
$$s_1s_2x_0=s_1^{-1}s_2x_0= n_1^{-1}.(n_2.x_0).$$
Thus,
$\E_1|_{s_1s_2x_0}$ is obtained from $\h_{\gamma_1}$ by applying first
$n_2$ (and obtaining  $\h_{\gamma_3}$), followed by $n_1^{-1}$ (which gives $\h_{\gamma_2}$).
Consequently,
$$ e_M\left(\E_1|_{s_1s_2x_0} \right)=e_M(\h_{\gamma_2})=b_2.$$

All other restriction formulae can be  proved similarly.
\end{proof}

The following lemma expresses $e_M(\E_3)$ in terms of $e_M(\E_1)$ and
$e_M(\E_2)$.

\begin{lemma}\label{wehave} We have
$$b_3=b_1+b_2$$
and
$$e_M(\E_3)=e_M(\E_1)+e_M(\E_2).$$
\end{lemma}

\begin{proof}  We take
Equation (\ref{ege}) and restrict both sides to $s_1x_0$.
The left-hand side changes according to   Lemma \ref{ther}. The right-hand side doesn't change. Indeed, for any $k\in \{1,2,3\}$ we have
$$P^*(b_k)|_{s_1x_0}=i_{s_1x_0}^*(P^*(b_k))=
i_{s_1x_0}^*(P^*(e_M(\h_{\gamma_k})))
=(P\circ i_{s_1x_0})^*(e_M(\h_{\gamma_k})),$$
which is the same as the $M$-equivariant Euler class of the pullback of
$\h_{\gamma_k}$ via the map
 $P\circ i_{s_1x_0}:\{s_1x_0\}\to \{x_0\}$; this is equal to $b_k$.
Equation (\ref{ege}) implies
$$-b_1+b_3-b_2=b_1+b_2-b_3,$$
which, in turn, implies the desired equations.
\end{proof}

We are now ready to show that the relations given by Equation  
(\ref{ege1}) hold true. For each of  them we restrict the left-hand side to
$x_0$, $s_1x_0$, $s_2x_0, \ldots, s_1s_2s_1x_0$ and use 
Lemmata \ref{ther} and \ref{wehave}; each time we do this, we obtain
$S_2(2b_1+b_2,-b_1+b_2,-b_1-2b_2)$, respectively
$S_3(2b_1+b_2,-b_1+b_2,-b_1-2b_2)$.
Indeed, let $S$ be one of the (symmetric) polynomials $S_2$ and $S_3$. We have
as follows:
\begin{align*}{}&S(2e_M(\E_1)+e_M(\E_2),-e_M(\E_1)+e_M(\E_2),-e_M(\E_1)-
2e_M(\E_2))|_{s_1x_0}\\
   =&S(-2b_1+b_3,b_1+b_3,b_1-2b_3)\\
   =&S(-b_1+b_2,2b_1+b_2,-b_1-2b_2)\\
  =&S(2b_1+b_2,-b_1+b_2,-b_1-2b_2).
\end{align*}
\begin{align*}{}&S(2e_M(\E_1)+e_M(\E_2),-e_M(\E_1)+e_M(\E_2),-e_M(\E_1)-
2e_M(\E_2))|_{s_2x_0}\\
   =&S(2b_3-b_2,-b_3-b_2,-b_3+2b_2)\\
   =&S(2b_1+b_2,-b_1-2b_2,-b_1+b_2)\\
  =&S(2b_1+b_2,-b_1+b_2,-b_1-2b_2).
\end{align*}
\begin{align*}{}&S(2e_M(\E_1)+e_M(\E_2),-e_M(\E_1)+e_M(\E_2),-e_M(\E_1)-
2e_M(\E_2))|_{s_1s_2x_0}\\
   =&S(2b_2-b_3,-b_2-b_3,-b_2+2b_3)\\
   =&S(-b_1+b_2,-b_1-2b_2,2b_1+b_2)\\
  =&S(2b_1+b_2,-b_1+b_2,-b_1-2b_2).
\end{align*}
\begin{align*}{}&S(2e_M(\E_1)+e_M(\E_2),-e_M(\E_1)+e_M(\E_2),-e_M(\E_1)-
2e_M(\E_2))|_{s_2s_1x_0}\\
   =&S(-2b_3+b_1,b_3+b_1,b_3-2b_1)\\
   =&S(-b_1-2b_2,2b_1+b_2,-b_1+b_2)\\
  =&S(2b_1+b_2,-b_1+b_2,-b_1-2b_2).
\end{align*}
\begin{align*}{}&S(2e_M(\E_1)+e_M(\E_2),-e_M(\E_1)+e_M(\E_2),-e_M(\E_1)-
2e_M(\E_2))|_{s_1s_2s_1x_0}\\
   =&S(-2b_2-b_1,b_2-b_1,b_2+2b_1)\\
   =&S(-b_1-2b_2,-b_1+b_2,2b_1+b_2)\\
  =&S(2b_1+b_2,-b_1+b_2,-b_1-2b_2).
\end{align*}

Our second goal is to  show that $e_M(\E_1)$ and $e_M(\E_2)$  generate $H^*_M(X)$ as an 
$H^*(BM)$-algebra. To this end we first recall that the action of $M$ on $X$ is equivariantly formal. 
From Equation (\ref{dimh}) we deduce that there exists a basis 
$\bar{\alpha}_0,\ldots, \bar{\alpha}_5$
of $H^*_M(X)$ over $H^*(BM)$, such that each
$\bar{\alpha}_k$ is a homogeneous element of degree given by
$$\deg \bar\alpha_k=
\begin{cases}
0, \ {\rm if \ }  k=0\\
8, \ {\rm if \ } k\in \{1,2\}\\
16, {\rm \ if \ } k\in \{3,4\}\\
24, {\rm \ if \ } k=5.
\end{cases}
$$
We need the following lemma.

\begin{lemma}  
There exists a basis $\{\tilde{\alpha}_k \mid k=0,\ldots, 5\}$
of $H^*_M(X)$ as an $H^*(BM)$-module such that:
\begin{itemize}
\item[(i)] if $k\in \{0,\ldots, 5\}$, then both $\tilde{\alpha}_k$ and 
$${\alpha}_k:=\jmath^*(\tilde{\alpha}_k)\in H^*(X)$$
 are homogeneous
of degree given by
$$\deg\tilde\alpha_k=\deg\alpha_k=\deg\bar\alpha_k$$
\item[(ii)] the set $\{\alpha_k \mid k=0,\ldots, 5\}$ is a basis
of $H^*(X)$ over $\bR$
\item[(iii)] we have
$$\tilde{\alpha}_{1}=e_M(\E_1), \ 
\tilde{\alpha}_{2}=e_M(\E_2),$$
and 
$${\alpha}_{1}=e(\E_1), \ 
{\alpha}_{2}=e(\E_2).$$
\end{itemize} 
\end{lemma}

\begin{proof} We 
set
$$\tilde{\alpha}_k
:=
\begin{cases}
\bar{\alpha}_k, \ {\rm if} \ k \neq 1, 2\\
e_M(\E_1), \ {\rm if } \ k=1\\
e_M(\E_2), \ {\rm if } \ k=2.
\end{cases}
$$
It is sufficient to show that 
\begin{align*}
{}&\bar{\alpha}_1=r_{11}e_M(\E_1)+r_{21}e_M(\E_2)+a_{1}\\
{}&\bar{\alpha}_2=r_{12}e_M(\E_1)+r_{22}e_M(\E_2)+a_{2}
\end{align*}
where $r_{11}, r_{21}, r_{12}, r_{22}$ are real numbers such that the matrix
$(r_{ij})_{1\le i,j\le 2}$ is non-singular and
$a_{1}$, $a_{2}$ are in $H^*(BM)$. 
Indeed, we have
$$\jmath^*(e_M(\E_1))=e(\E_1) \quad {\rm and} \quad 
 \jmath^*(e_M(\E_2))=e(\E_2).
$$
The cohomology classes $e(\E_1)$ and $e(\E_2)$ are a basis of $H^8(X)$ (see Section \ref{three}). Also $\jmath^*(\bar\alpha_1)$ and $\jmath^*(\bar\alpha_2)$ are a basis
of $H^8(X)$ (because $\ker \jmath^*=\langle H^+(BM).H^*_M(X) \rangle$).
Thus, we can write
\begin{align*}
{}&\jmath^*(\bar{\alpha}_1)=r_{11}\jmath^*(e_M(\E_1))+r_{21}\jmath^*(e_M(\E_2))\\
{}&\jmath^*(\bar{\alpha}_2)=r_{12}\jmath^*(e_M(\E_1))+r_{22}\jmath^*(e_M(\E_2))
\end{align*}
for some numbers $r_{11}, r_{21}, r_{12}, r_{22}$ such that the matrix
$(r_{ij})_{1\le i,j\le 2}$ is non-singular.
Consequently, the differences $\bar{\alpha}_1-r_{11}e_M(\E_1)-r_{21}e_M(\E_2)$
and $\bar{\alpha}_2-r_{12}e_M(\E_1)-r_{22}e_M(\E_2)$
are linear combinations with coefficients in $H^+(BM)$ of
$\bar\alpha_0,\ldots,\bar\alpha_5$. By dimension reasons,  both of them must live in
$H^+(BM)$.
This finishes the proof.
  \end{proof}

Let us now consider the isomorphism of $H^*(BM)$-modules
$$\Psi: H^*_M(X)\to H^*(X)\otimes H^*(BM)$$
given by
$\Psi(\tilde\alpha_k):=\alpha_k,$
for all $k=0,\ldots, 5$.

From now on we identify 
the $H^*(BM)$-algebra $H^*_M(X)$ with
$H^*(X)\otimes H^*(BM)$ equipped with the product
$\circ$. The latter is defined by the fact that it is $H^*(BM)$-bilinear and it satisfies 
the condition
$$\alpha_k\circ\alpha_\ell:=\Psi^{}(\tilde{\alpha}_k \tilde{\alpha}_\ell),$$
for all $k,\ell \in \{0,\ldots,5\}$.
We stress that 
\begin{equation}\label{iden}H^*_M(X)=(H^*(X)\otimes \bR[a_1,a_2,a_3,a_4], \circ)\end{equation}
as $\bR[a_1,a_2,a_3,a_4]$-algebras, see Equation (\ref{hgb}).
The usual grading of $H^*(X)$ together with
$$\deg a_1=4, \ \deg a_2=\deg a_3=8, \ \deg a_4=12,$$
induces a grading on $H^*(X)\otimes H^*(BM)$.
The following two properties of the product $\circ$ will be used 
later. If $\alpha,\beta \in H^*(X)$ are homogeneous elements, then we have: 
\begin{itemize}
\item[(i)] $\alpha \circ \beta $ is a homogeneous element of
$H^*(X)\otimes H^*(BM)$ 
 of degree given by $\deg ( \alpha \circ \beta) =\deg \alpha +\deg \beta$, 
\item[(ii)]  $\alpha \circ \beta= \alpha\beta +$(a linear combination of multiples of $H^+(BM)$).
\end{itemize}
Point (i) follows from the fact that the map  $\Psi$ is degree preserving. 
To justify point (ii) it is sufficient to take
$\alpha=\alpha_k$ and $\beta=\alpha_\ell$, where $k,\ell \in \{0,\ldots, 5\}$; we use the fact that the following diagram is commutative: 
$$H^*_M(X)\stackrel{\Psi}{\longrightarrow} H^*(X)\otimes H^*(BM)$$
$$\label{pstar} \jmath^* \searrow \ \ \ \ \ \ \  \ \ \ \  \swarrow \ \ \ \ \ \ \ \ \ \ \ \ \ \ \  
$$
$$H^*(X) \ \ \ \ \ \ \ \ \ \ \ \  $$
Here the arrow in the right-hand side is the canonical projection.

\begin{lemma}\label{be} The classes 
$$\ep_1:=e_M(\E_1) \quad {\it and} \quad   
\ep_2:=e_M(\E_2)$$
generate $H^*_M(X)$ as  an $H^*(BM)$-algebra.
Equivalently, in terms of the identification
(\ref{iden}), the classes
$$\ep_1=e(\E_1) \quad {\it and} \quad   
\ep_2 =e(\E_2)$$
generate $(H^*(X)\otimes H^*(BM),\circ)$ as  an $H^*(BM)$-algebra.
\end{lemma}

\begin{proof} 
It is sufficient to prove that for any $k\in \{0,\ldots, 5\}$,
$\alpha_k$   can be written as a polynomial expression in $\ep_1$ and $\ep_2$ with coefficients
in $H^*(BM)$, the product being $\circ$.  
We prove this by induction on $k$. 
The claim is obvious for $k=0$,
as $\alpha_0$ is just a number (element of $H^0(X)$).
  Let us now make the induction step:
  take $k\in\{0,\ldots,5\}$, $k\ge 1$. 
We know that $\ep_1$ and $\ep_2$  generate
  $H^*(X)$ (see Lemma \ref{allc}).
 Thus, we have
  $$\alpha_k = f(\ep_1,\ep_2),$$
  where $f$ is a polynomial in two variables and the product in 
  the right hand side is the usual (cup) product.
  Let $ f^{\circ}(\ep_1,\ep_2)$ be the element of $H^*(X)\otimes H^*(BM)$ 
  obtained by evaluating $f$ in terms of the product $\circ$.
  By property (ii) of $\circ$, $\alpha_k-f^{\circ}(\ep_1,\ep_2)$  is a linear combination of
  terms of the form $a. \alpha_\ell$, where $a\in H^+(BM)$ and
  $\ell \in \{0,\ldots, 5\}$ with $\deg \alpha_\ell < \deg \alpha_k$.
  The last condition implies $\ell <k$:
  we only need to use the induction hypothesis.
\end{proof}


The following lemma will finish the proof of Theorem \ref{main}.
\begin{lemma}\label{ideal} The ideal of relations in 
$H^*_M(X)$ with respect to  
$\ep_1$ and $\ep_2$ is generated by  
 (\ref{ege1}).
\end{lemma}

\begin{proof} 
Let us consider the polynomials $g_2, g_3\in \bR[x_1,x_2]$ given by
\begin{equation}g_i=S_i(2x_1+x_2,-x_1+x_2,-(x_1+2x_2))\end{equation}
$i=2,3$.
We prove that if $f(x_1,x_2)\in H^*(BM)\otimes \bR[x_1,x_2]$ such that
\begin{equation}\label{f}f^{\circ}(\ep_1,\ep_2)=0,\end{equation}
then $f$ is in the ideal generated by the  polynomials
\begin{align*}f_i(x_1,x_2):=g_i(x_1,x_2)-g_i(b_1,b_2),\end{align*}
$i=2,3$ (here  $f^\circ(\ep_1,\ep_2)$ is the element of $H^*(X)\otimes H^*(BM)$ obtained by evaluating  $f(x_1,x_2)$ on $\ep_1$, $\ep_2$ in the ring
$(H^*(X)\otimes H^*(BM),\circ)$).
We prove this claim by induction on
$\deg f$: throughout this proof, {\it degree} will always be considered
only with respect to $x_1$ and $x_2$.  If $\deg f=0$ then the claim is obvious.
Let us now perform the induction step.
We consider a non-constant polynomial $f(x_1,x_2)$ as above,
satisfying Equation (\ref{f}).
Let $h(x_1,x_2)$ be the component of $f(x_1,x_2)$ of highest degree
(with respect to $x_1,x_2$).
From the fact that $f^{\circ}(\ep_1,\ep_2)=0$ 
 and property (ii) of $\circ$  we deduce that
 $$h(\ep_1,\ep_2)=0,$$
 the product involved in the left-hand side being the usual (cup) product. 
 By Lemma \ref{allc},  $h(x_1,x_2)$ is a combination with
 coefficients in $H^*(BM)\otimes \bR[x_1,x_2]$ 
 of $g_2(x_1,x_2)$ and
 $g_3(x_1,x_2)$.
 We come back to Equation (\ref{f}) and
 replace $h$ by the expression mentioned in the previous sentence,
 where we complete each occurence of $g_i$ 
 to $f_i$ (by adding and subtracting the necessary quantity). 
 The cancellations which we obtain
 allow us to obtain another condition of type
 (\ref{f}), this time with a polynomial $f$ of
 degree strictly smaller than the previous one.  
Finally, we use the induction hypothesis.
\end{proof}

\begin{rem}\label{tria2} Like in Remark \ref{tria1}, we denote by ${\mathcal E}^-_3$ the vector
bundle $\E_3$ with the reversed orientation. Set 
\begin{align*}{}&\tilde{x}_1:=\frac{1}{3}(e_M(\E^-_3)-e_M(\E_2)), \ \tilde{x}_2:=\frac{1}{3}(e_M(\E_1)-e_M(\E^-_3)),  \ \tilde{x}_3:=\frac{1}{3}(e_M(\E_2)-e_M(\E_1)),   \\
{}& u_1:=\frac{1}{3}(-b_3-b_2), \  u_2:=\frac{1}{3}(b_1+b_3), \ u_3:=\frac{1}{3}(b_2-b_1).
\end{align*}
Theorem \ref{main} can be  rephrased by saying that $H^*_M({\rm Fl}(\O))$ is generated as an
$H^*(BM)$-algebra by $\tilde{x}_1, \tilde{x}_2, \tilde{x}_3$, subject to the following relations:
$$S_i(\tilde{x}_1, \tilde{x}_2, \tilde{x}_3)= S_i(u_1, u_2, u_3), \ i=2,3.$$
A similar description holds for $H^*_T({\rm Fl}_3(\bC))$, see Appendix \ref{lasts}. 
\end{rem}

\begin{rem}  Another presentation of the ring $H^*_M(X)=H^*_{{\rm Spin}(8)}({\rm F}_4/{\rm Spin}(8))$  can  be
deduced from \cite[Corollary 5.10]{Ho-Sj}, since ${\rm F}_4$ and ${\rm Spin}(8)$ have the same rank.   
\end{rem}

\section{Presentations of Goresky-Kottwitz-MacPherson type}

\subsection{A theorem of Harada, Henriques, and Holm}\label{sec:hhh}

Motivated by the well-known result of Goresky, Kottwitz, and MacPherson, see \cite{Go-Ko-Ma}, concerning
the  equivariant cohomology of complex projective varieties that are acted on by tori, Harada, Henriques, and Holm
considered in \cite{Ha-He-Ho} actions of arbitrary topological groups along with equivariant cohomology theories
associated to them. They obtained descriptions of the corresponding (cohomology) rings for spaces 
equipped with a certain stratification. We will confine ourselves here to state a weaker version 
of their main result, which is strictly what we need  
in order to prove Theorems \ref{princ} and \ref{kth}. 
If $M$ is an arbitrary  compact connected Lie group and $X$ a space acted on by $M$, we denote by $E^*_M(X)$ the 
$M$-equivariant cohomology ring with real coefficients or the $M$-equivariant complex topological $K$-theory ring of $X$:
 the result is valid for both $H^*_M(\cdot)$ and $K^*_M(\cdot)$.
 
  \begin{theo}\label{hhh} {\rm (\cite{Ha-He-Ho})} Let $$X=\bigsqcup_{k=1}^sC_k$$ be a finite CW complex
 whose open cells $C_k$, $1\le k\le s$, satisfy the following properties:
 \begin{itemize}
 \item[(i)] $C_k$ is an even dimensional real vector space
 equipped with an $M$-linear action with  a unique fixed point, say $p_k$,
 which is identified with $0$.
 \item[(ii)] We can decompose 
 \begin{equation}\label{ck} C_k=\bigoplus_{1\le \ell \le k}C_{k\ell},\end{equation}
 where $C_{k\ell}$ are vector subspaces (possibly equal to $\{0\}$) of $C_k$; the boundary
 $\partial_X (C_{k\ell})$ of $C_{k\ell}$ in $X$ consists of only one point, which is  fixed
by the $M$-action (in the case where $C_{k\ell}=\{0\}$, the fixed point is $p_k$).
\item[(iii)]  For any $k\in \{1,\ldots, s\}$, 
the equivariant Euler classes $e_M(C_{k\ell})$, where $1\le \ell \le k$ such that $C_{k\ell}\neq \{0\}$, 
are relatively prime elements of $E^*_M({\rm pt.})$ (here we regard $C_{k\ell}$ as a vector bundle
over a point).
   \end{itemize}
   Then the map $\imath^*:E_M^*(X)\to E^*_M(X^M)$ induced by the 
   inclusion of the $M$-fixed point set $X^M$ into $X$ is injective. Moreover, 
   the image of $\imath^*$ consists of all 
   $$(f_k)\in E^*_M(X^M)=\prod_{k=1}^sE^*_M({\rm pt.})$$ 
   such that
   $f_k-f_\ell$ is divisible by $e_M(C_{k\ell})$ for all $1\le \ell< k\le s$ with
   $C_{k\ell}\neq \{0\}$.  
  \end{theo}

\subsection{The CW complex structure of ${\rm Fl}(\O)$}\label{five}

We aim to apply Theorem \ref{hhh} to the special case of 
$X:={\rm Fl}(\O)$ and $M:={\rm Spin}(8)$. In this section we are concerned with 
assumptions (i) and (ii) in that theorem. More precisely, recall that $$ 
X:={\rm Fl}(\O)
={\rm F}_4.x_0,$$
where $x_0={\rm Diag}(x_1^0,x_2^0, x_3^0)$,
with $x_1^0,x_2^0,x_3^0\in \bR$, any two distinct, 
such that $x_1^0+x_2^0+x_3^0=0$;
this time we also assume that 
$x_2^0<x_3^0<x_1^0$, i.e., $\gamma_1(x^0)>0$ and $\gamma_2(x_0)>0$. 
 We use the CW decomposition of ${\rm Fl}(\O)$
 described by Equation (\ref{cwdec}).  That is, we choose $C_k:=C_\sigma$, for $\sigma\in \Sigma_3$.
The splitting (\ref{ck}) is the one described by Equation (\ref{cs}). Assumption (i) in Theorem
\ref{hhh} follows from Proposition \ref{fixp} and the fact that
 $C_\sigma\cap \Sigma_3x_0=\{\sigma x_0\}$. 
 For assumption (ii), we will need the explicit embedding of $C_\sigma$ in
 $X$, as given in Theorem \ref{dkv} (b). That is, we consider the root spaces
 $\g_{\gamma} \subset \eg_{6(-26)}$, where $\gamma\in \Phi$, as well as the diffeomorphism
 $\sum\g_\gamma \to X$, $x\mapsto \exp(x)(\sigma x_0)$, where the sum in the domain runs over all $\gamma\in \Phi^+$ such that $\sigma^{-1}\gamma \in \Phi^-$. 
 Assumption (ii) follows readily from the following lemma.
 
 \begin{lem} If $\sigma$ and $\gamma$ are as above, then the boundary  of 
 $\exp(\g_\gamma)(\sigma x_0)$ in $X$ is $\{s_\gamma \sigma x_0\}$.
 \end{lem} 
 
\begin{proof} Let us consider  again $\Phi=\{\pm \gamma_1,\pm \gamma_2,\pm\gamma_3\}$, which is a root system of type
$A_2$. The corresponding Weyl group $W$, see  (\ref{weylg}),  is isomorphic to $\Sigma_3$.
It contains the reflections $s_{\gamma_i}$ about $\ker \gamma_i$, $i=1,2,3$.
Each of those is a transformation of $\d^0$ which permutes the three coordinates of any vector
in the following  way:
 \begin{equation}\label{sgama} s_{\gamma_1}=(2,3), \ 
s_{\gamma_2}=(1,3), \ s_{\gamma_3}=(1,2).\end{equation} 
Here, as usual, $(i,j)$ denotes the $i,j$
transposition in $\Sigma_3$. In fact, $W$  is generated by 
$s_1:=s_{\gamma_1}$ and $s_2:=s_{\gamma_2}$.
The precise description of $W$ is given by Equation (\ref{wey=}).
We will need the following table, which gives for every $\sigma \in \Sigma_3$ the set of all 
$\gamma \in \Phi^+=\{\gamma_1,\gamma_2,\gamma_3\}$ such that $\sigma^{-1}\gamma\in \Phi^-$.  

\begin{tabular}{|l|l|}
	\hline
$\sigma$& $\gamma$   \\
	\hline
$s_1$& $\gamma_1$    \\

$s_2$ & $\gamma_2$     \\

$s_2s_1$ & $\gamma_2$, $\gamma_3$     \\

$s_1s_2$& $\gamma_1$, $\gamma_3$     \\	

$s_1s_2s_1$& $\gamma_1$, $\gamma_2$, $\gamma_3$ \\
\hline
\end{tabular}

\noindent Here we have used the formulae:
$s_k(\gamma_k)=-\gamma_k$ for $k=1,2$,
$s_1(\gamma_2)=s_2(\gamma_1)=\gamma_3$,
$s_1(\gamma_3)=\gamma_2$, and $s_2(\gamma_3)=\gamma_1$.

 \begin{figure}[h]
\begin{center}
\epsfig{figure=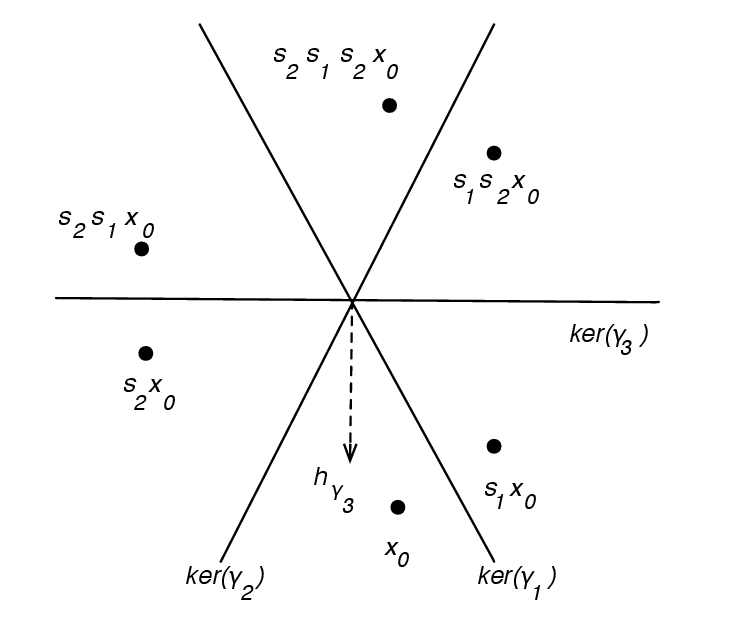, width=3.5in}
\end{center}
\begin{center}
{\rm Figure 1.}
\end{center}
\end{figure}

Let us discuss in detail the following two situations. 

\noindent {\it Case 1.} $(\sigma,\gamma)=(s_1,\gamma_1)$. 
We need to  show that the boundary of $\exp(\g_{\gamma_1})(s_1 x_0)$ is $x_0$.
To this end, we note that $\exp(\g_{\gamma_1})(s_1 x_0)$ is a Schubert cell
(see Appendix \ref{lasts1}).
Thus, by \cite[Section 4, especially Equation (4.10)]{Du-Ko-Va}, 
its closure 
consists of the cell itself together with the 0 dimensional cell $\{x_0\}$.

\noindent {\it Case 2.} $(\sigma,\gamma)=(s_1s_2s_1, \gamma_3)=(s_{\gamma_3},\gamma_3)$. We now show that the boundary of $\exp(\g_{\gamma_3})(s_3 x_0)$ is $\{x_0\}$. To simplify notations, we set  
$$G:={\rm E}_{6(-26)}, \ K:={\rm F}_4,  \ \g:=\eg_{6(-26)}, \ 
\k:=\fg_4,  \ \s := \h_3^0(\O),  \ {\rm and} \  \gamma_3:=\gamma.$$
As usual, we denote $M={\rm Spin}(8)$. We also denote by $N$ and $A$  the connected Lie
subgroups of $G$ of Lie algebras $\g_{\gamma_1}+\g_{\gamma_2}+ \g_{\gamma_3}$,
respectively $\a$ 
(the notations above have been used in the general case in Appendix \ref{lasts1}).
We will use the rank-one reduction procedure, as described in \cite[Ch.~IX, Section 2]{He}. Let us denote by $\g^{\ga}$ the Lie subalgebra of $\g$ generated by 
$\g_{\ga}$ and $\g_{-\ga}$. 
 Take $h_{\ga} \in\a$ determined by
$\langle h_{\ga} , h \rangle=\gamma(h)$, for all $h\in\a$
(here $\langle \ , \ \rangle$ is the Killing form of $\g$). We have the Cartan decomposition
$$\g^{\ga} = \k^{\ga} \oplus \s^{\ga},$$
where $\k^{\ga} = \k \cap \g^{\ga}$ and $\s^{\ga} = \s \cap \g^{\ga}$
(see also Equation (\ref{card})).
The space $\bR h_{\ga}$ is maximal abelian in $\s^{\ga}$. Let $G^{\ga}$,
$K^{\ga}$, and $A^{\ga}$ denote the   connected Lie subgroups of 
$G$ of Lie algebras $\g^{\ga}, \k^{\ga}$, respectively $\bR h_{\ga}$. 
Then we have $K^{\ga} = K\cap G^{\ga}$ and $A^\ga =A\cap G^\ga$. 
Moreover, if $M^{\ga}$ denotes the centralizer
 of $h_{\ga}$ in $K^{\ga}$, then we have
 $M^{\ga}=M\cap G^{\ga}$.  The connected Lie subgroup  of
 $G^\ga$ of Lie algebra $\g_\ga$ is $N^\ga = G^\ga \cap N$. The Iwasawa
 decomposition of $G^\ga$ is
 $$G^\ga = K^\ga A^\ga N^\ga.$$
 Without loss of generality we may assume that 
 $x_0=h_{\ga}$: since the last two vectors are in the same  Weyl chamber
 (see Figure 1), their $K$-orbits  are $G$-equivariantly diffeomorphic.  
 Consequently, we have $s_{\gamma_3}x_0=-h_\ga$. 
The orbit $X^\ga:=K^{\ga}.h_{\ga}$ is contained in $X=K.h_{\ga}$ (for both orbits, the group
action is the Adjoint one). 
In fact, the inclusion is $G^{\ga}$-equivariant. Indeed, the action of $G$ on $X$ is induced by the
identification $X=G/MAN$. Consequently, the subgroup  $G^{\ga}$ of $G$
acts on $X$ and the orbit of the coset of $e$  is $G^\ga/(MAN\cap G^\ga)=G^\ga/(M^\ga A^\ga N^\ga)=X^\gamma$
(here we have used that the map $K\times A \times N \to G$,
$(k,a,n)\mapsto kan$, for all $(k,a,n)\in K\times A \times N$ is a diffeomorphism).
The Schubert cell decomposition of $X^\ga$ described in Theorem \ref{dkv} (b) is
$$X^\ga=\exp (\g_\ga)(-h_\ga) \bigsqcup \{ h_\ga\}.$$
Thus, the  cell $\exp (\g_\ga)(-h_\ga)$  is dense in $X^\ga$. We deduce that the closure of
$\exp (\g_\ga)(-h_\ga)$ in $X$ is equal to $X^\ga$. This finishes the proof.

The other cases follow immediately from the two above. For instance, to show that the boundary of 
$\exp(\g_{\gamma_1})(s_1s_2x_0)$ is $s_1(s_1s_2x_0)=s_2x_0$ we use Case 2. 
Indeed, we replace  $x_0$ by $s_2x_0$ and $s_1,s_2$ by $s_2,s_3$ respectively
 (reflections about the walls of the Weyl chamber which contains $s_2x_0$). 
\end{proof}

\subsection{The root structure of ${\rm Spin}(8)$}\label{rootspi}

It remains to verify assumption (iii) in Theorem \ref{hhh} for the cell decomposition (\ref{cwdec})
and the splittings (\ref{cs}) in two situations: equivariant cohomology and equivariant $K$-theory. 
This will be done by calculating explicitly the corresponding Euler classes. 
We  need the following description of the roots,
weights, and of the representation ring of ${\rm Spin}(8)$. The details can be found for instance in \cite[Ch.~V, Section 6 and Ch. ~VI, Section 6]{Br-tD}
or \cite[Ch.~4]{Ad}. 
The Lie algebra of ${\rm Spin}(8)$ is the space $\mathfrak{so}(8)$ of all
skew-symmetric $8\times 8$ matrices  whose entries are real numbers.
Recall that  by $T$ we have denoted the canonical maximal torus of ${\rm Spin}(8)$ (see Proposition \ref{eac}). 
Its Lie algebra, call it $\t$,  consists of all matrices 
of the form
 $$
\left(
\begin{array}{ccccccccccccc}
0 & \theta_1 & & & \\
-\theta_1 & 0 &  & \ \  0   & \\
 &  & \ddots &   \\
 & 0 \ \ \ \ &  &  0 & \theta_4\\
 & &  &  -\theta_4 &0
\end{array}
\right),
$$ 
where $\theta_1,\theta_2,\theta_3,\theta_4\in \bR$. 
For any $j\in \{1,2,3,4\}$ we denote by $L^j$  the linear function on $\t$ which assigns to
each matrix of the form above the number $\theta_j$. The set $\{L^1, L^2, L^3, L^4\}$ is a basis of the dual space $\t^*$.

\begin{itemize}
\item The set of roots of ${\rm Spin}(8)$ with respect to
$\t$ is $$\Phi_{{\rm Spin}(8)} = \{ \pm L^{i}\pm L^{j}  \mid  1 \leq i < j \leq 4 \} .$$
\item A simple root system is 
$$\Pi = \{  L^1 - L^2 ,    L^2 - L^3,  L^{3} - L^{4},  L^3 +L^4\}.$$
The corresponding set of positive roots is 
$$\Phi_{{\rm Spin}(8)}^+ = \{ L^{i}\pm L^{j}  \mid 1 \leq i < j \leq 4 \} .$$
\item 
The corresponding fundamental weights are:   
\begin{equation}\label{weight}\rho_1 = L^1, \ \rho_2 = L^1 +L^2, \ \rho_3 = \frac{L^1+L^2+L^3-L^4}{2}, \   \rho_4 = \frac{L^1+L^2+L^3+L^4}{2}.\end{equation}
Since ${\rm Spin}(8)$ is simply connected, these weights are a basis of the  lattice $\mathfrak{t}_{\mathbb{Z}}^*$ of integral forms.
We will also use the presentation 
$$\mathfrak{t}_{\mathbb{Z}}^*=\oplus_{1 \leq i \leq 5 }\mathbb{Z}\omega^i /(2 \omega^5- \omega^1 -  \omega^2 - \omega^3 -  \omega^4),$$ where we have denoted as follows:
$$\omega^1 := L^1, \ \omega^2 := L^2, \ \omega^3 := L^3, \  \omega^4 := L^4, \ \omega^5:= \frac{L^1+L^2+L^3+L^4}{2}.$$
As usual, to any integral form $ \lambda \in \mathfrak{t}_{\mathbb{Z}}^*$ corresponds the
character $e^{\lambda} \in R[T]$.
In this way, if we denote  $y_j:=e^{\omega^{j}}$,  $1 \leq j \leq 5$, we obtain the following presentation: $$R[T] \cong\mathbb{Z}[y_1^{\pm1}, y_2^{\pm1}, y_3^{\pm1}, y_4^{\pm1}, y_5^{\pm1}]/(y_5^2 - y_1y_2y_3y_4).$$

\item The canonical action of the Weyl group $W_{{\rm Spin}(8)}=N_{{\rm Spin}(8)}(T)/T $
on $\t^*$ is faithful. The linear automorphisms  of $\mathfrak{t}^*$ induced in this way 
are those   $\eta$ with the property that for any $1\leq i \leq 4$, there exists $1\leq j \leq 4$ 
such that $\eta(L^{i})=\pm L^{j}$, the number of ``$-$" signs being even.

\item The  representation ring of ${\rm Spin}(8)$ is $R[{\rm Spin}(8)]=\mathbb{Z}[X_1,X_2,X_3,X_4]$
where
$$X_1=V_8\otimes \bC,  \ X_2=S_8^+\otimes \bC, \ X_3=S_8^-\otimes \bC,$$
and $X_4$ is the complexified adjoint representation of ${\rm Spin}(8)$
(recall that $V_8$ is induced by the standard representation of ${\rm SO}(8)$ on $\bR^8$
via the covering map ${\rm Spin}(8)\to {\rm SO}(8)$ and $S_8^{\pm}$ are the real half-{spin} representations of
${\rm Spin}(8)$). 
Their weights are as follows  (see \cite[Proposition 4.2]{Ad}):
\begin{itemize}
\item[(i)] For $X_1$:
\begin{equation}\label{x1} \pm L^1, \pm L^2, \pm L^3, \ {\rm and } \pm L_4. \end{equation} 
\item[(ii)] For $X_2$:
\begin{equation}\label{x3}  \frac{\pm L^1\pm L^2\pm L^3\pm L^4}{2},
\end{equation}
where
the number of ``$-$" signs is even.
\item[(iii)] For $X_3$:
\begin{equation}\label{x2}\frac{\pm L^1\pm L^2\pm L^3\pm L^4}{2},
\end{equation} where
the number of ``$-$" signs is odd.
\item[(iv)] For $X_4$:
all  roots of ${\rm Spin}(8)$ relative to $T$.
\end{itemize}
The (complex) dimension of each weight space is equal to 1.
Consequently, the restriction/inclusion map $ R[{\rm Spin}(8)] =R[T]^{W_{{\rm Spin}(8)}}
\rightarrow R[T] $ is given by 
\begin{align*}X_1 =& y_1+ y_1^{-1}  +  y_2 + y_2^{-1} +  y_3 + y_3^{-1}+  y_4 + y_4^{-1} \\ 
 X_2 =& y_5 +  y_5y_1^{-1}y_2^{-1} +  y_5y_1^{-1}y_3^{-1} + y_5y_1^{-1}y_4^{-1} +
 y_5y_2^{-1}y_3^{-1} + y_5y_2^{-1}y_4^{-1}+ y_5y_3^{-1}y_4^{-1} \\ 
{}&+ y_5y_1^{-1}y_2^{-1}y_3^{-1}y_4^{-1}\\
 X_3 =& y_5y_1^{-1} +  y_5y_2^{-1} +  y_5y_3^{-1} + y_5y_4^{-1} +
 y_5y_1^{-1}y_2^{-1}y_3^{-1} + y_5y_1^{-1}y_2^{-1}y_4^{-1} + y_5y_1^{-1}y_3^{-1}y_4^{-1} \\{}&+ y_5y_2^{-1}y_3^{-1}y_4^{-1}\\
 X_4 =& \sum_{1 \leq i < j \leq 4} y_i^{\pm 1} y_j^{\pm 1}.
 \end{align*}
 
\end{itemize}

 Recall that the $T$-actions on $V_8$ and $S_8^{\pm}$ are $\bC$-linear relative
 to certain complex linear structures on these spaces, see Proposition \ref{eac}.
 We would like now to calculate the weights of each of these three $T$-modules.
 For $V_8$ they are
$L_1$, $L_2$, $L_3$, and $L_4$. 
The ${\rm Spin}(8)$-module  $S_8^+$ differs from $V_8$ by  a group automorphism of ${\rm Spin}(8)$,
see \cite[Ch.~14]{Ad}. 
This is just one of the outer automorphisms that arise from the many symmetries of the
Dynkin diagram of ${\rm Spin}(8)$. It leaves $T$ invariant and the induced automorphism 
of $\t$ is the reflection $s_{\omega_5-\omega_4}$ through $\ker (\omega_5-\omega_4)$ (equip $\t$ with the inner product
which makes $(\theta_1, \theta_2, \theta_3, \theta_4)$ into an orthonormal coordinate system).
Thus, the weights  of $S_8^+$  are: 
\begin{align*}
{}& s_{\omega_5-\omega_4}(L_1)=\omega_1+\omega_4-\omega_5= \rho_1-\rho_3 \\
{}& s_{\omega_5-\omega_4}(L_2)=\omega_2+\omega_4-\omega_5= -\rho_1+\rho_2-\rho_3\\ 
{}& s_{\omega_5-\omega_4}(L_3)=\omega_3+\omega_4-\omega_5=-\rho_2+\rho_4 \\
{}& s_{\omega_5-\omega_4}(L_4)=\omega_5= \rho_4.
\end{align*} 
A similar reasoning holds for $S_8^-$, the automorphism of $\t$ being this time 
$s_{\omega_5}$. The resulting weights are:
\begin{align*}{}&s_{\omega_5}(L_1) =\omega_1-\omega_5=\rho_1-\rho_4  \\
{}& s_{\omega_5}(L_2)=\omega_2-\omega_5=-\rho_1+\rho_2-\rho_4 \\
{}& s_{\omega_5}(L_3)=   \omega_3-\omega_5=\rho_3-\rho_2 \\
{}& s_{\omega_5}(L_4)=   \omega_4-\omega_5=-\rho_3.
\end{align*}

\subsection{Equivariant cohomology of ${\rm Fl}(\O)$} 
We can now calculate the Euler classes $e_M(C_{k\ell})$ mentioned in Theorem \ref{hhh}
for $M={\rm Spin}(8)$ and $X={\rm Fl}(\O)$.  They are the ${\rm Spin}(8)$-equivariant Euler classes of
$V_8$ and $S_8^\pm$. If $V$ is any of these representations, then we can split
$V=\bigoplus_{i=1}^4 \ell_i$, where $\ell_i$ are 1-dimensional  $T$-invariant complex vector subspaces,
see Proposition \ref{eac}. 
Consequently,
$$e_T(V)=c_4^T(\bigoplus_{i=1}^4\ell_i)=c_1^T(\ell_1)\cdots c_1^T(\ell_4),$$
where $c_4^T$ and $c_1^T$ denote the $T$-equivariant Chern classes.
We know that the 1-dimensional complex representations
of $T$ are labeled by the character group ${\rm Hom}(T,S^1)$, and
the map ${\rm Hom}(T,S^1)\to H^2(BT;\bZ)$ given by $L\mapsto c_1^T(L)$ is  
a group isomorphism (see for example \cite[Ch.~20, Section 11]{Hu}). 
In turn, ${\rm Hom}(T,S^1)$ is isomorphic to the  lattice of integral forms on $\t$.
The formulae obtained at the end of Subsection \ref{rootspi} thus give readily descriptions of $e_T(V_8)$ and $e_T(S_8^\pm)$
as elements of $H^*_T({\rm pt.}) = H^*(BT)= S(\t^*)$. 
On the other hand, there is a  canonical inclusion 
$H^*(B{\rm Spin}(8)))\hookrightarrow H^*(BT)$, which maps ${\rm Spin}(8)$-equivariant to $T$-equivariant  Euler classes.
We deduce:
\begin{align}
\label{ermsp1}{}&e_{{\rm Spin}(8)}(V_8)=  
\rho_1(-\rho_1+\rho_2)(-\rho_2+\rho_3+\rho_4)(-\rho_3+\rho_4)
\\
\label{ermsp2}{}&e_{{\rm Spin}(8)}(S_8^+)=(\rho_1-\rho_3)(-\rho_1+\rho_2-\rho_3)(-\rho_2+\rho_4)\rho_4
\\
{}&e_{{\rm Spin}(8)}(S_8^-)=
-(\rho_1-\rho_4)(-\rho_1+\rho_2-\rho_4)(\rho_3-\rho_2)\rho_3.
\end{align}

Since $H^*_{{\rm Spin}(8)}({\rm pt.})=H^*(B{{\rm Spin}(8)})=\bR[\rho_1, \rho_2,\rho_3, \rho_4]^{W_{{\rm Spin}(8)}}$, we have shown:

\begin{lem}\label{lem}   The equivariant
Euler classes $e_{{\rm Spin}(8)}(V_8)$, $e_{{\rm Spin}(8)}(S_8^-)$, and $e_{{\rm Spin}(8)}(S_8^+)$  
are pairwise relatively prime elements
 of $H^*_{{\rm Spin}(8)}({\rm pt.})$.
 \end{lem}

 We are now in a position to prove Theorem \ref{princ}.

 \noindent{\it Proof of Theorem \ref{princ}.} From Theorem \ref{hhh}
 we deduce that the map $\imath^*:H^*_M({\rm Fl}(\O))\to H^*_M(\Sigma_3x_0)=
 \prod_{\sigma\in\Sigma_3}H^*(BM)$ is injective.
 Moreover, its image consists of those $(f_\sigma)_{\sigma\in\Sigma_3}$ with the following property:
 
 \noindent {\it (P1) }
 $f_\sigma-f_{s_\gamma\sigma}$ is divisible by
$e_M(\h_\gamma)$ for any  $\sigma \in \Sigma_3$ and any $\gamma \in\Phi^+$ such that $\sigma^{-1}\gamma \in\Phi^-$. 

\noindent Condition (P1) is equivalent to:  

\noindent {\it (P2)} $f_\sigma-f_{s_\gamma\sigma}$ is divisible by
$e_M(\h_\gamma)$ for any  $\sigma \in \Sigma_3$ and any  $\gamma \in\Phi^+$. 

\noindent Indeed, (P2) implies (P1). Also (P1) implies (P2):
assume that (P1) holds true and take $\sigma\in\Sigma_3$ and $\gamma\in\Phi^+$ such that $\sigma^{-1}\gamma \in\Phi^+$;
then we have $s_\gamma(s_\gamma\sigma) =\sigma$ and also 
$(s_\gamma \sigma)^{-1}\gamma=-\sigma^{-1}\gamma,$
which is in $\Phi^-$; thus, by (P1), the difference $f_{s_\gamma\sigma}-f_\sigma$ is divisible by 
$e_M(\h_\gamma)$. 

Finally, recall that $\tb_{i}-\tb_j=e_M(\h_{\gamma_k})$, where
$1\le i<j\le 3$ and $\{k\}=\{1,2,3\}\setminus \{i, j\}$ (see Equations (\ref{eulercl}),  (\ref{conse}), (\ref{ek}), and
(\ref{letusdef}), as well as  Proposition \ref{foral} and Lemma 
\ref{wehave}). 
We also take  into account Equation (\ref{sgama}).
\hfill $\square$ 

The next lemma is relevant for the observation made right after Theorem \ref{princ}.
First recall from the introduction that $H^*(BM)=\bR[a_1,a_2,a_3,a_4]$, where  $a_1\in H^4(BM)$, $a_2,a_3\in H^8(BM)$,
 and $a_4\in H^{12}(BM)$.

\begin{lem} \label{lasr} The elements $a_1^2$, $\tb_1$, and $\tb_2$ of $H^8(BM)$ are linearly
independent. Consequently, we have
$$H^*(BM)=\bR[a_1, \tb_1, \tb_2, a_4].$$
\end{lem} 

\begin{proof} As before, we regard $H^*(BM)$ as a subspace of $H^*(BT) = \bR[\rho_1,\rho_2,\rho_3,\rho_4]$. 
One can see that $\tb_1$ and $\tb_2$ are, up to a possible negative sign, just the Euler classes
$e_{{\rm Spin}(8)}(V_8)$ and  $e_{{\rm Spin}(8)}(S_8^+)$, respectively. Concretely, these are given by
Equations (\ref{ermsp1}) and (\ref{ermsp2}), respectively: observe that those two polynomials are linearly independent.
Assume now that there exists a linear combination of them which is equal to $a_1^2$.  
Recall that $a_1$, regarded as a (polynomial) function on $\t$, is nothing but the norm squared of a vector.
In particular, the only zero of $a_1$ is for $\rho_1=\rho_2=\rho_3=\rho_4=0$. 
On the other hand, the polynomials given by (\ref{ermsp1}) and (\ref{ermsp2}) vanish 
whenever $\rho_1=\rho_2=\rho_3=0$ and $\rho_4$ is arbitrary. This is a contradiction.\end{proof} 

\subsection{Equivariant $K$-theory of ${\rm Fl}(\O)$}\label{lastsec}

Our aim here is to prove Theorem \ref{kth}. Unlike in the previous section, we will apply 
the Harada-Henriques-Holm theorem for the $T$-action, rather than the ${\rm Spin}(8)$-action.
We then take into account that $K_{{\rm Spin}(8)}({\rm Fl}(\O))=K_T({\rm Fl}(\O))^{W_{{\rm Spin}(8)}}$.

The first step in the proof is made by 
calculating the $K$-theoretical  $T$-equivariant Euler classes of $C_\sigma$, see Equation (\ref{cwdec}).
By (\ref{cs}), this amounts to calculating  $e^K_T(V_8), e^K_T(S_8^-)$, and $e^K_T(S_8^+)$, 
where the superscript $K$ stands for $K$-theory.
Recall that the Euler class of a direct sum is the product of the
Euler classes of the summands; also, if a torus acts on $\bC$ with weight $\lambda$, then
the resulting equivariant $K$-theoretical Euler class is $1-e^{-\lambda}$ (see e.g.~\cite[Note, p.~35]{Bott}).
From the expressions of the weights which we have obtained at the end of Subsection \ref{rootspi} we
obtain:
\begin{align*}
{}& e_{T}^K(V_8) = (1-y_1^{-1}) (1-y_2^{-1}) (1-y_3^{-1})(1-y_4^{-1})\\
{}& e_{T}^K(S_8^+) =(1-y_1^{-1}y_4^{-1}y_5)(1-y_2^{-1}y_4^{-1}y_5)(1-y_3^{-1}y_4^{-1}y_5)
(1-y_5^{-1})\\
{}& e_{T}^K(S_8^-) = (1-y_1^{-1}y_5)(1-y_2^{-1}y_5)(1-y_3^{-1}y_5)(1-y_4^{-1}y_5).
\end{align*}
We immediately deduce:

\begin{lem} The $K$-theoretical equivariant Euler classes $e^K_T(V_8), 
e^K_T(S_8^+)$, and  $e^K_T(S_8^-)$ are pairwise relatively prime elements
of $R[T]$.
\end{lem}

Thus, the $T$-action on ${\rm Fl}(\O)$ satisfies the  hypotheses of Theorem \ref{hhh}
for $K^*_T(\cdot)$  (also recall that, by Proposition \ref{fixp}, the $T$-fixed point set is $\Sigma_3x_0$).
We deduce that $K_T^1({\rm Fl}(\O))=\{0\}$, as well as the following result, which concerns 
$K_T^0({\rm Fl}(\O))$. 

 \begin{propo}\label{canhom}  The ring homomorphism
 $\imath_T^*:K_T({\rm Fl}(\O)))\to K_T(\Sigma_3 x_0)=\prod_{\sigma\in \Sigma_3}R[T]$ 
 induced by the inclusion map $\imath : \Sigma_3x_0 \to {\rm Fl}(\O)$ is injective. Its image  consists of
 all $(f_\sigma) \in \prod_{\sigma \in \Sigma_3}   \mathbb{Z}[y_1^{\pm1}, y_2^{\pm1}, y_3^{\pm1}, y_4^{\pm1}, y_5^{\pm1}]/(y_5^2 - y_1y_2y_3y_4)$ with the  property that for any $\sigma \in \Sigma_3$ we have: 
\begin{itemize}
\item[(i)] $ f_{(2,3)\sigma}-f_\sigma
 {\it  \ is  \ divisible \ by \ } 
 (1-y_1^{-1}) (1-y_2^{-1}) (1-y_3^{-1})(1-y_4^{-1}) $
\item[(ii)] $f_{(1,3)\sigma}-f_\sigma {\it  \ is  \ divisible \ by \ } 
(1-y_1^{-1}y_4^{-1}y_5)(1-y_2^{-1}y_4^{-1}y_5)(1-y_3^{-1}y_4^{-1}y_5)(1-y_5^{-1})
$
\item[(iii)] $ f_{(1,2)\sigma}-f_\sigma {\it  \ is  \ divisible \ by \ } (1-y_1^{-1}y_5)(1-y_2^{-1}y_5)(1-y_3^{-1}y_5)(1-y_4^{-1}y_5)$.
\end{itemize}
The divisibility referred to above is in the ring $R[T]=\mathbb{Z}[y_1^{\pm1}, y_2^{\pm1}, y_3^{\pm1}, y_4^{\pm1}, y_5^{\pm1}]/(y_5^2 - y_1y_2y_3y_4)$.
\end{propo}

We are now ready to accomplish the main goal of the subsection:

\noindent {\it Proof of Theorem \ref{kth}.}  Let $\imath^*_M : K_{{\rm Spin}(8)}({\rm Fl}(\O)) \to
\prod_{\sigma \in \Sigma_3}R[{\rm Spin}(8)]$ be the ring homomorphism  induced by the
inclusion $\imath: \Sigma_3x_0 \to {\rm Fl}(\O)$. The map $\imath^*_M$  is obviously 
$W_{{\rm Spin}(8)}$-equivariant, where the action of $W_{{\rm Spin}(8)}$ on 
$\prod_{\sigma \in \Sigma_3}R[{\rm Spin}(8)]$ is the diagonal one. 
 We have the ring isomorphisms 
 \begin{equation}\label{laspin}R[{\rm Spin}(8)] \cong R[T]^{W_{{\rm Spin}(8)}} \quad {\rm and} \quad K_{{\rm Spin}(8)}({\rm Fl}(\O)) \cong K_T({\rm Fl}(\O))^{W_{{\rm Spin}(8)}},\end{equation} where for the last one we 
invoke \cite[Corollary 4.10 (ii)]{Ha-Sj}.  
Since $\imath_T^*$ is $W_{{\rm Spin}(8)}$-equivariant, we may 
identify the two pairs of spaces related by the isomorphisms (\ref{laspin}) and assume that $\imath^*_M$ is just the restriction of $\imath^*_T$
to $K_T({\rm Fl}(\O))^{W_{{\rm Spin}(8)}}$. Consequently, the image of $\imath^*_M$ is the intersection
of the image of $\imath^*_T$ with $\prod_{\sigma \in \Sigma_3}R[{\rm Spin}(8)]$.
It thus consists of all $(f_\sigma)$ in the latter direct product which satisfy the
divisibility properties (i), (ii), and (iii) in Proposition \ref{canhom}.   
Recall now that $R[{\rm Spin}(8)] =\bZ[X_1, X_2, X_3, X_4]$. From the explicit formulae for
$X_1, X_2,$ and $X_3$ given in Subsection \ref{rootspi} we deduce by direct calculation:
\begin{align*}
{}& X_2-X_3 = y_5(1- y^{-1}_1 )  (1-y^{-1}_2)  (1-y^{-1}_3 )  (1-y^{-1}_4)\\
{}& X_1-X_3=y_4(1-y_1^{-1}y_4^{-1}y_5)(1-y_2^{-1}y_4^{-1}y_5)(1-y_3^{-1}y_4^{-1}y_5)(1-y_5^{-1})\\
{}& X_1-X_2 = -y_5^{-1}(1-y_1^{-1}y_5)(1-y_2^{-1}y_5)(1-y_3^{-1}y_5)(1-y_4^{-1}y_5).
\end{align*}
Consequently, for $(f_\sigma) \in \prod_{\sigma \in \Sigma_3} \bZ[X_1, X_2, X_3, X_4]$,
conditions (i), (ii), and (iii) in Proposition \ref{canhom} are equivalent to:
\begin{itemize}
\item[(i')] $ f_{(2,3)\sigma}-f_\sigma
 {\rm \ is  \ divisible \ by \ } 
 X_2-X_3 $
\item[(ii')] $f_{(1,3)\sigma}-f_\sigma {\rm  \ is  \ divisible \ by \ } 
X_1-X_3$
\item[(iii')] $ f_{(1,2)\sigma}-f_\sigma {\rm  \ is  \ divisible \ by \ } X_1-X_2$.
\end{itemize}
This finishes the proof.
\hfill $\square$

\appendix

\section{The complex flag manifold
${\rm Fl}_3(\bC)$}\label{lasts}

This section is a recollection of  well-known facts
concerning the complex flag manifold ${\rm Fl}_3(\bC)$.
The focus is of course on those aspects  whose 
 counterparts in the realm of octonions we  deal with in this paper. 
 Our aim here is to smooth the passage from complex numbers to octonions.

Originally,  ${\rm Fl}_3(\bC)$ is the set of all nested sequences
$$V_1\subset V_2 \subset \bC^3,$$
where $V_1$ and $V_2$ are complex vector subspaces of $\bC^3$ such that
$\dim V_1=1$ and $\dim V_2=2$. 
Alternatively, let us equip $\bC^3$ with the standard Hermitian inner product:
then ${\rm Fl}_3(\bC)$  is the set of pairs $(L_1,L_2)$, where
$L_1$ and $L_2$ are 1-dimensional complex vector subspaces
of $\bC^3$ with $L_1$ orthogonal to $L_2$. 

Let us now consider the space 
$$\h_3(\bC)
=\{a\in {\rm Mat}^{3\times 3} (\bC) 
 \mid a=a^*\}.$$
We have the following presentations.
\begin{proposition}
a) There is a natural identification between the complex projective plane
$\bC {\rm P}^2$ and the set of all matrices
$a\in \h_3(\bC)$ with
$$a^2=a \ {\it and } \ {\rm tr}(a)=1.$$

b) There is a natural identification between the flag manifold ${\rm Fl}_3(\bC)$ and the set of
all pairs $(a_1,a_2)\in \bC {\rm P}^2\times \bC {\rm P}^2$ with the property that
$${\rm Re}({\rm tr}(a_1a_2))=0.$$
\end{proposition}

The identifications are as follows.

\noindent {\it For $\bC {\rm P}^2$.}
 A 1-dimensional complex vector subspace $V$ of $\bC^3$ 
 is identified with the element of $\h_3(\bC)$ which has eigenvalues 
 $1,0,0$ and 1-eigenspace equal to $V$ (the $0$-eigenspace is implicitly
 $V^{\perp}$). 
 Moreover, an element $a$ of $\h_3(\bC)$ has eigenvalues
 $(1,0,0)$ if and only if
 $a^2=a$ and ${\rm tr}(a)=1$. 
 
 \noindent {\it For ${\rm Fl}_3(\bC)$.} Take $L_1, L_2$ two 1-dimensional
 complex vector subspaces of $\bC^3$ and
 $a_1,a_2$ the Hermitian matrices with eigenvalues
 $(1,0,0)$ and $1$-eigenspaces $L_1$, respectively $L_2$.
 The main point is that $L_1$ is perpendicular to  $L_2$ if and only if
 ${\rm Re (tr}(a_1a_2))=0.$
   Indeed, let us  choose an orthonormal basis
   $v_1,v_2,v_3$ of $\bC^3$, where $L_2=\bC v_1$. Then we have:
\begin{align*}{\rm tr}(a_1a_2)
&= \langle a_1a_2(v_1), v_1\rangle
+ \langle a_1a_2(v_2), v_2\rangle 
+ \langle a_1a_2(v_3), v_3\rangle\\
{}&=  \langle a_1(v_1), v_1\rangle= \langle a_1^2(v_1), v_1\rangle
= \langle a_1(v_1), a_1^*(v_1)\rangle
= \langle a_1(v_1), a_1(v_1)\rangle.
\end{align*}
Thus, ${\rm Re}({\rm tr}(a_1a_2))=0$ if and
and only if $a_1(v_1)=0$. 
On the other hand,  $L_1$ is perpendicular to $L_2$ if and only if
$L_2$ is contained in the $0$-eigenspace of $a_1$, that is,
$a_1(v_1)=0.$ 

There are three natural projections ${\rm Fl}_3(\bC) \to \bC {\rm P}^2$: the first maps
an arbitrary pair $(L_1, L_2)$ to $L_1$, the second to $L_2$, and the third to the orthogonal
complement of $L_1\oplus L_2$ in $\bC^3$. These are all three 
$\bC {\rm P}^1$-bundles. By taking the tangent space  to the fiber at any point, one obtains
three subbundles of the tangent bundle of ${\rm Fl}_3(\bC)$, which we denote by
$\E_1^c, \E_2^c,$ and $\E_3^c$. These are complex line bundles over
${\rm Fl}_3(\bC)$. On the other hand, ${\rm Fl}_3(\bC)$ has also three tautological
bundles, call them $\L_1, \L_2, \L_3$. The following identifications are natural:
$$\E_1^c = \L_2\otimes \L_3^*, \quad  \E_2^c = \L_3\otimes \L_1^*,
\quad \E_3^c = \L_2\otimes \L_1^*.$$They induce the following 
relationship concerning the first Chern classes:
$$c_1(\E_1^c) = c_1(\L_2)-c_1(\L_3), \quad  c_1(\E_2^c) = c_1(\L_3)-c_1(\L_1),
\quad  c_1(\E_3^c) = c_1(\L_2)-c_1(\L_1).$$
Consequently, we have
$$c_1(\E^c_3)=c_1(\E^c_1)+ c_1(\E^c_2).$$
But $\L_1\oplus \L_2\oplus \L_3$ is a trivial  vector bundle over ${\rm Fl}_3(\bC)$, hence
$$c_1(\L_1) +c_1(\L_2) + c_1(\L_3)=0.$$
We are led to:
\begin{align}
{}&c_1(\L_1) = -\frac{1}{3}(c_1(\E^c_1) + 2c_1(\E^c_2)) \nonumber\\
{}& c_1(\L_2) = \frac{1}{3}(2c_1(\E_1^c)+c_1(\E_2^c)) \label{chern} \\ 
{}& c_1(\L_3)= \frac{1}{3}(-c_1(\E_1^c)+c_1(\E_2^c)).\nonumber \end{align}
A theorem of Borel, see \cite{Bo}, says that the ring $H^*({\rm Fl}_3(\bC);\bZ)$ is generated by
$x_1:=c_1(\L_1), x_2:=c_1(\L_2),$ and $x_3:=c_1(\L_3)$, subject to the relations given by the vanishing
of all symmetric polynomials in $x_1, x_2, $ and $x_3$. We deduce:
\begin{proposition} The ring  $H^*({\rm Fl}_3(\bC);\bZ)$ is generated by
$\frac{1}{3}(2c_1(\E_1^c)+c_1(\E_2^c))$ and $\frac{1}{3}(c_1(\E^c_1)+2c_1(\E^c_2))$,
subject to the relations
$$S_i\left(\frac{1}{3}(2c_1(\E^c_1)+c_1(\E^c_2)),\frac{1}{3}(-c_1(\E^c_1)+c_1(\E^c_2)), 
-\frac{1}{3}(c_1(\E^c_1)+2c_1(\E^c_2))\right)=0$$
$i=2,3$. 
\end{proposition}
Let us now denote by
$T$ the 3-torus which consists of all diagonal matrices of the form ${\rm Diag}(z_1, z_2, z_3)$, 
$z_i\in \bC$, $|z_i|=1$. The natural splitting $T=S^1\times S^1\times S^1$
induces $BT=BS^1\times BS^1\times BS^1$, hence 
$H^*(BT)=\bR[u_1]\otimes \bR[u_2]\otimes \bR[u_3]=\bR[u_1, u_2, u_3]$.
There is a canonical action of $T$ on ${\rm Fl}_3(\bC)$. The corresponding $T$-equivariant
cohomology ring $H^*_T({\rm Fl}_3(\bC))$ can also be described by a Borel type
formula. Concretely, as an $H^*(BT)$-algebra, $H^*_T({\rm Fl}_3(\bC))$ is generated by
the $T$-equivariant first Chern classes of $\L_1$, $\L_2$, and $\L_3$, which are:
$$\tilde{x}_1:=c_1^T(\L_1), \quad  \tilde{x}_2:=c_1^T(\L_2), \quad \tilde{x}_3:=c_1^T(\L_3).$$ The ideal of relations is generated by:
$$S_j(\tilde{x}_1, \tilde{x}_2, \tilde{x}_3)=S_j(u_1, u_2, u_3),$$
$j=1,2,3$ (the details are spelled out for instance in \cite[Proof of Theorem 1.1]{Ma2}). 
This time, instead of (\ref{chern}) we have:
\begin{align}
{}&c^T_1(\L_1) = -\frac{1}{3}(c^T_1(\E^c_1) + 2c^T_1(\E^c_2))+\frac{1}{3}(u_1+u_2+u_3) \nonumber\\
{}& c^T_1(\L_2) = \frac{1}{3}(2c^T_1(\E_1^c)+c^T_1(\E_2^c))+\frac{1}{3}(u_1+u_2+u_3) \nonumber \\ 
{}& c^T_1(\L_3)= \frac{1}{3}(-c^T_1(\E_1^c)+c^T_1(\E_2^c))+\frac{1}{3}(u_1+u_2+u_3).\nonumber
 \end{align}
Set
$$b_1:=u_2-u_3 \ {\rm and} \ b_2:= u_3-u_1.$$
We have proved:
\begin{proposition} As an $H^*(BT)$-algebra,  $H^*_T({\rm Fl}_3(\bC))$ is generated
by $c_1^T(\E^c_1)$ and $c_1^T(\E^c_2)$,
subject to the relations:
\begin{align*}
{}&S_i(2c_1^T(\E^c_1)+c_1^T(\E^c_2),-c_1^T(\E^c_1)+c_1^T(\E^c_2), -(c_1^T(\E^c_1)+2c_1^T(\E^c_2))) \\
 {}&  \ \ \ \ \  \ \ \ \ \ \ \ \ \ \  \ \ \ \ \ \ \ \ \ \  \ \ \ \ \ \ \ \ \ \  \ \ \ \ \ \ \ \ \ \  \ \ \ \ \ \ \ \ \ \ =S_i(2b_1+b_2,-b_1+b_2,-(b_1+2b_2)),
\end{align*}
$i=2,3$.
\end{proposition}

The $T$-fixed points in ${\rm Fl}_3(\bC)$ are all pairs of type $(\bC e_i, \bC e_j)$,
$1\le i, j\le 3$, $i\neq j$. Here $e_i$ is the $i$-th coordinate vector in $\bC^3$.
Thus, we have a natural identification ${\rm Fl}_3(\bC)^T\cong \Sigma_3$.
One can show that the restriction map $H^*_T({\rm Fl}_3(\bC))\to H^*_T(\Sigma_3)$ is injective and
its image consists of all $(f_\sigma)\in \prod_{\sigma\in \Sigma_3}\bR[u_1, u_2, u_3]$
such that
$f_\sigma -f_{(i,j)\sigma}$ is divisible by $u_i-u_j$, for all $\sigma \in \Sigma_3$ and $1\le i<j\le 3$:
this is the standard Goresky-Kottwitz-MacPherson type description of $H^*_T({\rm Fl}_3(\bC))$,
which can be immediately deduced from the original work \cite{Go-Ko-Ma}. 

A similar description holds for the $T$-equivariant $K$-theory ring of ${\rm Fl}(\O)$.
Namely, let us first identify $R[T]=\bZ[t_1^{\pm 1}, t_2^{\pm 1}, t_3^{\pm 1}]$. 
Then the restriction map 
$K_T({\rm Fl}_3(\bC))\to K_T(\Sigma_3)$ is injective and its image consists of
all $(f_\sigma)\in \prod_{\sigma\in \Sigma_3}\bZ[t_1^{\pm 1}, t_2^{\pm 1}, t_3^{\pm 1}]$ 
with the property that $f_\sigma -f_{(i,j)\sigma}$ is divisible by $1-t_it_j^{-1}$,
 for all $\sigma \in \Sigma_3$ and $1\le i<j\le 3$. This is a direct application of \cite[Theorem 1.7]{Mc}
 (cf.~also the appendix of \cite{Kn-Ro}).

\section{Real flag manifolds and their cell decomposition}\label{lasts1}

In this section we will present some general notions and results
concerning real flag manifolds. The main reference is  \cite{Du-Ko-Va} (the
background material  can be found for instance in \cite[Ch.~IX]{He}).

Let $G$ be a real connected semisimple Lie group and denote by $\g$ its Lie algebra.
Let $$\g=\k\oplus \s$$ be a Cartan decomposition: this means that the Killing form of
$\g$ is strictly negative definite on $\k$ and strictly positive definite on $\s$.
The corresponding Cartan involution is $\theta: \g\to\g$,
$$\theta(x+y)=x-y,$$
for all $x\in \k$ and $y\in \s$.
We pick a maximal abelian subspace $\a$ of $\s$ and consider the following root space decomposition:
\begin{equation}\label{rootdec}
\g=\m+\a+\sum_{\gamma\in \Phi}\g_{\gamma}.
\end{equation}
Here $\m$ is the centralizer of $\a$ in $\k$ and $\Phi$ the set of roots, which are 
functions $\gamma :\a\to \bR$ such that the root space
$$\g_\gamma:=\{x\in \g \mid [h,x]=\gamma(h)x \ {\rm for  \ all \ } h\in \a\}$$
is non-zero.
The set $\Phi$ is a root system in   the dual space $\a^*$. Let us pick a 
system of simple roots and denote by $\Phi^+$ the corresponding set of
positive roots. We set
$$\n:=\sum_{\gamma\in \Phi^+}\g_\gamma$$ and
obtain the Iwasawa decomposition
$$\g=\k\oplus \a \oplus \n.$$
If $K$, $A$, $N$ are the connected Lie subgroups of $G$ of Lie algebras
$\k$, $\a$, and $\n$ respectively, then we have the following Iwasawa decomposition of $G$:
$$G=KAN.$$
 Let us also denote by $M$ the centralizer of $\a$ in $K$ and by $W$ the Weyl group,
which is
$$W=\{k\in K \mid {\rm Ad}_Gk(\a)\subset \a\}/M.$$

It turns out that, via the adjoint representation of $G$, the group $K$ leaves $\s$ invariant. The orbits of the resulting representation are called {\it real flag manifolds}.
We  need the following result:

\begin{proposition}\label{x00} Take $x_0\in \a$ such that $\gamma(x_0)\neq 0$ for all
$\gamma\in \Phi$. Then the stabilizer of $x_0$ in $K$ is equal to $M$.
\end{proposition}

\begin{proof} By \cite[Proposition 1.2]{Du-Ko-Va}  the stabilizer
$K_{x_0}$ of $x_0$ satisfies
$$K_{x_0}=MK_{x_0}^0,$$
where $K_{x_0}^0$ denotes the identity component of $K_{x_0}$.
The Lie algebra of $K_{x_0}$ is the commutator of $x_0$ in $\k$.  
From the root decomposition (\ref{rootdec}), this is the same as $\m$.
Thus, we have $K_{x_0}^0\subset M$ and consequently 
$K_{x_0}=M$.
\end{proof}
 
Consequently, we can identify
$$X:={\rm Ad}_G(K)x_0=K/M.$$ 
From this we can see that there is a canonical embedding of the Weyl group $W$ in
$X$.

The natural action of $K$ on $X$ extends to an action of $G$. This arises from the identification
$$X=K/M=KAN/MAN = G/MAN,$$
where we take into account that $MAN$ is a subgroup of $G$.
The cell decomposition of $X$ we will describe in the following theorem   uses
 the embedding $W\subset X$ and also  the action of $G$ on $X$. The  proof can be found in \cite[Section 3]{Du-Ko-Va}.

\begin{theorem}{\rm (\cite{Du-Ko-Va})}\label{dkv}
(a) We have
\begin{equation}\label{cwde}X=\bigsqcup_{w\in W} Nw.\end{equation}

(b) Fix $w\in W$. The map $\sum \g_\gamma \to Nw$,
$x\mapsto \exp (x)w $ is a diffeomorphism. The sum in the domain
runs over all $\gamma\in \Phi^+$ such that $w^{-1}\gamma \in \Phi^-$.

(c) The decomposition (\ref{cwde}) makes $X$ into a  CW complex.
\end{theorem}

The cells $Nw$, $w\in W$, are usually referred to as {\it Schubert cells}. 

Let us now consider the following root space decomposition of $\s$:
\begin{equation}\label{rooa}\s=\a+\sum_{\gamma\in \Phi^+}\s_\gamma,\end{equation}
where $$\s_{\gamma}=(\g_\gamma+ \g_{-\gamma})\cap \s=
\{x\in \s \mid [h,[h,x]]=\gamma(h)^2x \ {\rm for  \ all \ } h\in\a\}.$$
We can easily see that both $\g_\gamma$ and $\s_\gamma$ are 
$M$-invariant, where $M$ acts via the Adjoint representation. The following result seems to be known.
Since we didn't find it clearly stated and proved in the literature, we 
included a proof of it.

\begin{proposition}\label{isomor} If $\gamma \in \Phi^+$, then the map $\Theta: \g_\gamma\to \s_\gamma$, given by $\Theta(x)=x-\theta x,$
for all $x\in \g_\gamma$, is an $M$-equivariant linear isomorphism.
\end{proposition}  

\begin{proof} First, since $\theta(\g_\gamma)=\g_{-\gamma}$, 
$\Theta$ is well defined, in the sense that it maps $\g_\gamma$ to $\s_\gamma$.
The map is also injective:  $x-\theta x=0$ and $x\in\g_\gamma$ implies $x=0$.
The map is also surjective: if $y\in \s_\gamma$, then we can write it as
$y=y_1+y_2,$ with $y_1\in\g_\gamma$
and $y_2\in\g_{-\gamma}$; since $y\in \s$, we have $\theta(y)=-y$, which implies
$y_2=-\theta(y_1)$, thus
$y=y_1-\theta(y_1)=\Theta(y_1)$.
The $M$-equivariance of $\Theta$ follows from the $M$-equivariance of $\theta$.
\end{proof}

We now take into account  that the map described in Theorem \ref{dkv} (b) is
$M$-equivariant, where $M$ acts on the domain by the Adjoint representation and
on the codomain via the $G$-action on $X$. We deduce:

\begin{corollary}\label{cwd} Fix $w\in W$. We have an $M$-equivariant diffeomorphism between the
Schubert
cell $Nw$ and the space $\sum \s_{\gamma},$
where the sum runs over all $\gamma\in \Phi^+$ such that
$w^{-1}\gamma\in \Phi^-$.
\end{corollary} 


\section{The symmetric space ${\rm E}_{6(-26)}/{\rm F}_4$}\label{lasts2}

In this section we will outline the construction of the (non-compact) symmetric space
mentioned in the title. We will try to make more clear several aspects mentioned in Subsection \ref{rfm}.
For instance, we will prove that the root spaces $\s_\gamma$ in the 
 decomposition described by Equation (\ref{rooa}) 
are the $\h_\gamma$ described in Subsection \ref{rfm}. This is an important fact,
because it allows us to deduce the presentation of $C_\sigma$ given by Equation (\ref{cs})
from Theorem \ref{dkv} (b) and Proposition \ref{isomor}.
The main reference of this section is the article \cite{Fr} by Freudenthal.

Recall that by Definition \ref{gpf}, the group of all linear transformations 
of $\h_3(\O)$ which preserve the product $\circ$  is ${\rm F}_4$. We   define the determinant function on $\h_3(\O)$ as follows:
$$\det (a)= \frac{1}{3}{\rm tr} (a\circ a \circ a)-\frac{1}{2}{\rm tr}  (a\circ a) {\rm tr}  a +\frac{1}{6}({\rm tr}  a)^3,$$ 
for all $a\in\h_3(\O)$. Let us consider the group of all linear transformations  of 
$\h_3(\O)$ which leave the determinant invariant. It turns out that this group is
just ${\rm E}_{6(-26)}$ (see Subsection \ref{rfm} for the definition of this group).
From the formula of the determinant above we deduce easily that
${\rm E}_{6(-26)}$ contains ${\rm F}_4$. Less obvious is that the latter group is a
maximal compact subgroup of the former.
The Lie algebra $\fg_4$ consists of all linear transformations of $\h_3(\O)$ of
the form $$\tilde{b} : \h_3(\O)\to\h_3(\O), \ 
\tilde{b}(y)=[b,y],$$
where $b$ is a $3\times 3$ matrix with entries in $\O$ such that $b=-b^*$
(that is, $b$ is skew-Hermitian).
Here and everywhere else in this section $[ \ , \ ]$ denotes the usual matrix commutator.
To any $a\in \h_3^0(\O)$ we attach the $\bR$-linear transformation
$\hat{a}$ of $\h_3(\O)$ given by
$$\hat{a}: \h_3(\O)\to \h_3(\O), \ \hat{a}(y)= a\circ y, \ {\rm for \ all}  \ y \in \h_3(\O).$$
The  Cartan decomposition of $\eg_{6(-26)}$ corresponding to $\fg_4$ is described in the following proposition (see \cite[end of Section 8.1.1]{Fr}).  

\begin{proposition}
If $c$ is in the Lie algebra $\eg_{6(-26)}$, then there 
exists $a\in \h_3^0(\O)$ and $b$ a $3\times 3$ skew-Hermitian matrix with entries in
$\O$ such that  $c=\tilde{b}+\hat{a}$. The matrices $a$ and $b$ are uniquely determined by $c$.
\end{proposition}
We see from here that $\eg_{6(-26)}=\fg_4\oplus \h_3^0(\O)$ is a Cartan decomposition, as already mentioned  in Subsection \ref{rfm}.

Note that the elements of $\eg_{6(-26)}$ are linear endomorphisms of $\h_3(\O)$.
We denote the Lie bracket by $[ \ , \ ]_*$: it is given by the commutator  
of the endomorphisms. We need the following lemma:

\begin{lemma} If $a,x \in \h_3^0(\O)$, then:
\begin{align*}
{}&  (i) \  [\hat{x},\hat{a}]_*=\frac{1}{4}\widetilde{[x,a]}\\
{}& (ii)  \ [\hat{x},[\hat{x},\hat{a}]_*]_* =\frac{1}{4}\widehat{[x,[x,a]]}.
\end{align*}
\end{lemma}

\begin{proof} (i) For any $y\in \h_3(\O)$ we have
$$[\hat{x},\hat{a}]_*(y)=\hat{x}(\hat{a}(y))-\hat{a}(\hat{x}(y))
=x\circ (a \circ y)- a\circ(x\circ y)
= \frac{1}{4}[[x,a],y].$$

(ii) For any $y\in \h_3(\O)$ we have
\begin{align*}{}&4[\hat{x},[\hat{x},\hat{a}]_*]_*(y)
=[\hat{x}, \widetilde{[x,a]}]_*(y)
=\hat{x}(\widetilde{[x,a]}(y))-
\widetilde{[x,a]}(\hat{x}(y))\\
{}& \ \ \ \ \ \ \ \ \ \ \ \ \ \  \ \ \ \ \ 
=x\circ ([[x,a],y])-[[x,a], x\circ y]
=[x,[x,a]]\circ y.
\end{align*}

\end{proof}

Let us now identify $\h_3^0(\O)$ with the subspace 
$\{\hat{x} \mid x\in \h_3^0(\O)\}$ of $\eg_{6(-26)}$. 
From Equation (ii) above we deduce that $\d^0$ is a maximal abelian subspace
of $\h_3^0(\O)$. By definition, a vector $a\in\h_3^0(\O)$ is a root vector with respect to a root
$\gamma$ if
$$[\hat{x},[\hat{x},\hat{a}]_*]_*=\gamma(\hat{x})^2\hat{a},$$
for all $x\in \d^0$. Again from Equation (ii) we deduce that  the roots of the symmetric space 
${\rm E}_{6(-26)}/{\rm F}_4$ with respect to $\d^0$  
are the functions $\frac{1}{2}(x_3-x_2)$, $\frac{1}{2}(x_1-x_3)$, and
$\frac{1}{2}(x_1-x_2)$ along with their negatives (see also Equation (\ref{gam})). The
corresponding root spaces
are the spaces $\h_{\gamma_1},\h_{\gamma_2}, $ and $\h_{\gamma_3}$  
described in Section \ref{rfm}.

\end{document}